\documentclass[12pt]{amsart}
\usepackage{amsmath,amscd,amsthm,amsfonts, amssymb,amsxtra}
\usepackage[all]{xy} \SelectTips{cm}{}
\xyoption{line}
\usepackage[dvipsnames]{xcolor}
\usepackage{hyperref}
  \hypersetup{colorlinks=true,citecolor=blue}
   \usepackage{graphicx}
   \usepackage[font=scriptsize]{caption}
   \usepackage{calrsfs}
 \usepackage[pagewise, displaymath, mathlines]{lineno}
  
\CDat
\newtheorem{thm}{Theorem}[section]  
\newtheorem*{un-no-thm}{Theorem}
\newtheorem{cor}[thm]{Corollary}     
\newtheorem{lem}[thm]{Lemma}         
\newtheorem{prop}[thm]{Proposition}  

\newtheorem{cl}[thm]{Claim}
\newtheorem{conjecture}[thm]{Conjecture}
\newtheorem{bigthm}{Theorem}
\newtheorem{bigcor}[bigthm]{Corollary}

\newtheorem{bigadd}[bigthm]{Addendum}

\newtheorem{bigconj}[bigthm]{Conjecture}

\theoremstyle{definition}
\newtheorem{defn}[thm]{Definition}   

\theoremstyle{definition}

\theoremstyle{definition}
\theoremstyle{remark}
\newtheorem{rem}[thm]{Remark}
\newtheorem{rems}[thm]{Remarks}

\newtheorem{term}[thm]{Terminology}
\newtheorem{hypo}[thm]{Hypothesis}
\newtheorem{notation}[thm]{Notation}
\newtheorem*{acks}{Acknowledgements}

\newtheorem*{out}{Outline}

\newtheorem*{intro-rem}{Remark}
\newtheorem*{intro-rems}{Remarks}

\newtheorem{ex}[thm]{Example}

\DeclareMathOperator*{\holim}{holim}

\DeclareMathOperator*{\hocolim}{hocolim}
\DeclareMathOperator*{\colim}{colim}

\DeclareMathOperator{\emb}{emb}

\begin{document}
\title[Homotopical Intersection Theory, III] {Homotopical Intersection Theory, III: multi-relative intersection problems}
\date{\today} 
\author{John R. Klein} 
\address{Wayne State University,
Detroit, MI 48202} 
\email{klein@math.wayne.edu} 
\author{Bruce Williams} 
\address{University of Notre Dame, Notre Dame, IN 46556}
\email{williams.4@nd.edu}
\begin{abstract} This paper extends some results of
Hatcher and Quinn \cite{H-Q} beyond the metastable range.
We give a bordism theoretic obstruction $\chi(f)$ to deforming
a map $f\colon\! P \to N$ between manifolds simultaneously off of a collection
of pairwise disjoint submanifolds $Q_1,...,Q_j \subset N$
under the assumption that $f$ can be deformed off of any proper subcollection in a homotopy coherent way. In a certain range of dimensions, $\chi(f)$ is a complete obstruction to finding the desired deformation. We apply this machinery 
to embedding problems and to the study of linking phenomena.
\end{abstract}
\thanks{The first author was partially supported by the 
National Science Foundation and the Simons Foundation.}
\maketitle
\setlength{\parindent}{15pt}
\setlength{\parskip}{1pt plus 0pt minus 1pt}
\def\Top{\bold T\bold o \bold p}
\def\wTop{\text{\rm w}\bold T}
\def\wT{\text{\rm w}\bold T}
\def\vo{\varOmega}
\def\vs{\varSigma}
\def\smsh{\wedge}
\def\flush{\flushpar}
\def\dbslash{/\!\! /}
\def\:{\colon\!}
\def\Bbb{\mathbb}
\def\bold{\mathbf}
\def\cal{\mathcal}
\def\orb{\cal O}
\def\hoP{\text{\rm ho}P}

\setcounter{tocdepth}{1}
\tableofcontents
\addcontentsline{file}{sec_unit}{entry}

\section{Introduction \label{intro}}

\subsection{Intersection problems}
In \cite{KW1} we considered the problem of deforming a map
$f\: P \to N$ between compact smooth manifolds off a compact smooth submanifold 
$Q\subset N$. This was called an {\it intersection problem}.
We obtained an obstruction $\chi(f)$ residing in a normal
bordism group $\Omega_0(X;\xi)$. The vanishing of the obstruction
is necessary for finding
such a deformation. One of the main results of 
\cite{KW1} was that in a certain metastable range of dimensions, $\chi(f)$ 
is a complete obstruction to finding a homotopy from $f$ to a map having
disjoint image from $Q$. 
The goal of the current paper is to extend these ideas  to the multi-relative setting.

Fix  a positive integer $j$ and
let 
\[
Q_1,\dots, Q_j \subset N
\] 
be a collection of pairwise
disjoint  closed smooth submanifolds of a compact connected smooth manifold $N$.
Given a map $P \to N$, where $P$ is a closed manifold,
the problem we consider is that of finding a deformation of $f$ off of the
$Q_i$ {\it simultaneously.}
We approach this inductively, by assuming that $P$ can be deformed off of
any proper union of the $Q_i$ in such a way that the deformations line up 
in a certain homotopically coherent fashion. We first explain what this precisely means.

Recall that a $(k+1)$-ad of spaces consists of
a space $X$ together with $k$ distinguished 
subspaces $X_1,\dots X_k \subset X$. The notation
for such data is $(X;X_1,\dots, X_k)$, but it will often be 
convenient to  simply write $X$ 
when the subspaces are understood.  

\begin{ex} (1). A space $Z$ can be
considered as a constant $(k+1)$-ad, i.e., $(Z;Z,\dots,Z)$.
\smallskip

\noindent (2). The standard $(k-1)$-simplex $\Delta^{k-1}$ together with its
codimension one faces is a $(k+1)$-ad, i.e., $(\Delta^{k-1};d_0\Delta^{k-1},\dots,
d_{k-1}\Delta^{k-1})$. 
\smallskip

\noindent (3).  If $Z$ is a space and $X$ is a $(k+1)$-ad, then
the cartesian product $Z \times X$ is a $(k+1)$-ad in the evident way.
\end{ex}

A map of $(k+1)$-ads $X \to Y$
is a continuous map of underlying spaces
which restricts to maps $X_i \to Y_i$ for all $i$.
We can topologize this as the subspace of the 
mapping space of all maps from $X$
to $Y$ in the compact-open topology. 

Consider $N$ together with the subspaces $N\setminus Q_1,\dots, N\setminus Q_j$
as a $(j+1)$-ad: $(N;N\setminus Q_1,\dots,N\setminus Q_j)$.  
Then  a 
{\it multi-relative intersection problem} is defined to be a map of $(j+1)$-ads 
\[
 f\: P\times \Delta^{j-1} \to N\, .
\]
Set $Q_J = Q_1 \amalg \cdots \amalg Q_j$.
We will consider $N \setminus Q_J$ as a constant $(j+1)$-ad; it
is then a sub-ad of $(N;N \setminus Q_1,\dots,N\setminus Q_j)$.
We define a {\it solution} to a multi-relative intersection problem 
to be a homotopy (of maps of 
$(j+1)$-ads) $f_t$  from $f = f_0$
to an ad map $f_1\: P\times \Delta^{j-1} \to N$ which factors 
as
\[
P \times \Delta^{j-1} @>>> N\setminus Q_J @>\subset >> N
\]
In particular, 
the image of $ f_1$ is disjoint from $Q_J$.

In more modern language the problem can be reformulated as
follows: let $J = \{1,\dots,j\}$. For $S \subset J$, let
\[
Q_S = \amalg_{i\in S} Q_i  \, .
\]
Then  a multi-relative intersection problem is equivalent to specifying 
a map
\begin{equation} \label{eqn:f-second}
f\:P \to \holim_{S \subsetneq J} (N\setminus Q_S) \, ,
\end{equation}
where the target is the homotopy inverse limit of the spaces
$N\setminus Q_S$ as $S$ ranges through the proper subsets of $J$.
Explicitly, the displayed homotopy limit is given by the space of maps of $(j+1)$-ads
$\Delta^{j-1} \to N$.

The deliberate ambiguity in our notation
is for the sake of convenience: we use
$f$ to denote the map \eqref{eqn:f-second} as well as for the map of ads
$P\times \Delta^{j-1} \to N$, as this is not likely to cause confusion 
(note: these maps determine each other by  an adjunction).

A solution then amounts to a map $\hat f\: P \to N\setminus Q_J$ together with 
a commuting homotopy $f_t\: P \to  \holim\limits_{S \subsetneq J} \, (N\setminus Q_S)$, 
$t\in [0,1]$,
for the diagram
\begin{equation} \label{eqn:lifting-problem}
\xymatrix{
& N\setminus Q_J \ar[d] \\
P \ar[r]_(.25){f} \ar[ur]^{\hat f} &
\holim\limits_{S \subsetneq J} \, (N\setminus Q_S)\, . \\
}
\end{equation}

Given a map of $(j+1)$-ads $f\: P \times \Delta^{j-1} \to N$ as above, we write
\[
E(P,Q_\bullet) 
\]
for the {\it iterated homotopy fiber product} of $P\times \Delta^{j-1}$ and
each of the $Q_i$ over $N$. This is just the homotopy pullback of 
the diagram 
\[
\CD
P\times \Delta^{j-1} \times  \prod_{i=1}^j Q_i @>>> \prod_{i=0}^{j} N @<\Delta << N
\endCD
\]
where $\Delta$ is the diagonal map, and the left  map
is the product of the map $f\: P \times \Delta^{j-1} \to N$ with the inclusions of the $Q_i$.

Define a virtual bundle $\xi$ over $E(P,Q_\bullet)$ as follows:
Let $\tau_P$ be the tangent bundle of $P$, $\tau_N$ the tangent
bundle of $N$ and $\tau_{Q_i}$ the tangent bundle of $Q_i$;
each one of these gives a bundle over $E(P,Q_\bullet)$
using the evident (projection) maps. To avoid notational clutter,
we use the same notation for these pullbacks. Then we set
\[
\xi := -\tau_P +  \sum_{i=1}^j(\tau_N- \tau_{Q_i}) \, .
\]
Suppose $p = \dim P$, $q_i = \dim Q_i$ and $n = \dim N$.
It will also be convenient to write 
\[
\mu = \min_i (n-q_i-2) \quad \text{and } \quad
\Sigma = \sum_i (n-q_i-2)\, .
\] In particular, the virtual
rank of $\xi$ is $2j-p+\Sigma$.  The following assumption will
be made throughout the paper.

\begin{hypo} For $1\le i \le j$, we have
$n-q_i \ge 2$.
\end{hypo}

We briefly review  the definition of
bordism with coefficients in a virtual bundle. Let $X$ be a space equipped
with a finite dimensional inner product bundle $\zeta$ of rank $s$. 
Then one has the Thom space $X^\zeta$ which is the
quotient space formed from the unit disk bundle by collapsing the unit sphere bundle to a
point. For the purposes of this paper, we define $\Omega_k(X;\zeta)$ to be the $k$-th stable homotopy group $\pi^{\text{\rm st}}_k(X^\zeta)$.  
By standard transversality arguments, an element of this abelian group is
represented by a compact smooth submanifold $V \subset \Bbb R^{k+d}$, for 
some $d \ge 0$,  together with a map
$g\:V \to X$ such that the pullback of $\zeta\oplus \epsilon^d$ along $g$  is identified with the normal bundle of $V$ (where $\epsilon^d$ is the trivial bundle of rank $d$; note that the dimension of $V$ is necessarily $k -s$). 
Then bordism defines an equivalence relation on this collection and the set of
equivalence classes is canonically identified with $\Omega_k(X;\zeta)$. With respect to this identification, note that the operation of disjoint union of bordism classes corresponds to the addition of stable homotopy classes. Now suppose that $\zeta$ is a virtual bundle. This means that $\zeta \oplus \epsilon^j$ comes equipped with an isomorphism to a finite dimensional inner product bundle $\eta$, for some  integer $j \ge 0$. In this instance, we define
$\Omega_k(X;\zeta)$ to be  $\Omega_{k+j}(X;\eta)$.  Our indexing convention for
the bordism group differs from that of \cite{KW1}, but is the same as the one used
in \cite{KW2}.

\begin{bigthm} \label{main-thm} Assume $j \ge 1$.
Then there is an obstruction
\[
\chi(f) \in \bigoplus_{(j-1)!} \Omega_{2j-2}(E(P,Q_\bullet);\xi)
\]
which vanishes if the intersection problem defined by $f$ possesses a solution.
Conversely, if
\[
p\le 1+\mu +\Sigma
\]
 then the vanishing
of $\chi(f)$ guarantees the existence of a solution. 
\end{bigthm}

Theorem \ref{main-thm} is proved using a fiberwise version of Poincar\'e duality
together with some general results about strongly cocartesian 
cubes.

\begin{rem} The $j=1$ case (``the metastable range'') of Theorem \ref{main-thm}
was already considered in \cite{KW1}. That work gave
a homotopy theoretic approach to the main results of the paper of Hatcher and Quinn
\cite{H-Q} (when $j=1$, Theorem \ref{embedding-result} below amounts to the vanishing
obstruction case of \cite[th.\ 2.2]{H-Q}). 
\end{rem}

\begin{rem}
The obstruction $\chi(f)$ is defined in a homotopy theoretic manner. Given
the identification between bordism theory
and the homotopy groups of a Thom spectrum,
it is reasonable to ask what $\chi(f)$ means geometrically.
In the $j=1$ case such an interpretation was provided by the ``Index Theorem'' of
\cite[th.~12.1]{KW1}.  The $j > 1$ case is more subtle and involves iterated intersections of null-bordism data. We hope to address this in detail another paper. 
Meanwhile, to leave the reader with an impression, 
we now sketch a  geometric description of $\chi(f)$ when $j=2$.

Let $j=2$ and let $f\:P \times \Delta^1 \to N$ be an intersection problem.
Let $b$ be the barycenter of $\Delta^1$ and let $D_i$ be 
the transversal intersection of $f_{|P \times b}\: P\times b\to N$ with $Q_i$.  By assumption, the evident maps $D_i \to E(P,Q_i)$ are null bordant. 
Let $g_i \: W_i \to E(P,Q_i)$ be a null-bordism.  
Compose this with the projection $E(P,Q_i) \to P$ to get maps $h_i \: W_i \to P$.
Now
take the transversal intersection of the product map
${h_1\times h_2} \: W_1 \times W_2 \to P\times P$
with the diagonal of $P$. This produces a closed manifold 
$W_{12}$ of dimension $p-2-\Sigma$ equipped with a map $W_{12} \to E(P,Q_\bullet)$ which is covered by 
the requisite bundle data. The associated bordism class 
coincides with the obstruction $\chi(f)$. 
\end{rem}

\begin{rem}[Large Codimension]\label{rem:large-codim} 
If $p \le 1+\Sigma$, then the bordism group
of Theorem \ref{main-thm} is trivial.
Consequently, $f$ can be homotopy factorized
through $N\setminus Q_J$ in this case.

If $j=1$, this
conclusion also follows from transversality,
and for $j >1$ it follows from the
higher Blakers-Massey  theorem applied to the $j$-cubical diagram
$\{N\setminus Q_S\}_{S\subset J}$ (cf.\ \cite[thm.~2.5]{Good}).
\end{rem}

\subsubsection{Highly connected manifolds}  
When the manifolds $P$ and $Q_i$ are sufficiently highly connected, the obstruction group
of Theorem \ref{main-thm} admits a simpler description.
Suppose that $P$ is $a$-connected and $Q_i$ is $b_i$-connected.  Choose basepoints
in $x\in P$ and $y_i \in Q_i$. Then $x$ gives rise to a point $x' \in N$ using $f$.
The homotopy fiber product of 
$E(x,y_\bullet)$ is defined and comes equipped with a map
$E(x,y_\bullet) \to E(P,Q_\bullet)$.  Moreover, the pullback of $\xi$ to
$E(x,y_\bullet)$ is a trivial virtual bundle of rank $2j-p+\Sigma$.
Hence the bordism groups associated with this pullback are framed
bordism groups of $E(x,y_\bullet)$ shifted in degree by $2j-p+\Sigma$.

It is also straightforward to check that
the map $E(x,y_\bullet) \to E(P,Q_\bullet)$ is $\min(a,b_1,\dots,b_j)$-connected.
It follows that the associated map of Thom spectra is 
$k$-connected, where $k = \min(a,b_1,\dots,b_j)+ 2j - p+\Sigma$. In particular
the induced homomorphism of bordism groups is an isomorphism in degrees
strictly less than $k$. 

Note that $E(x,y_\bullet)$ is the space of $j$-tuples
$(\lambda_1,\dots,\lambda_j)$ in which $\lambda_i\: [0,1] \to N$ is a path
from $x'$ to $y_i$ for $1\le i \le  j$.  The $j$-fold cartesian product of loop spaces
 $\prod_{j}\Omega N$
based at $x'$ acts on $E(x,y_\bullet)$ by path composition.
After a basepoint for $E(x,y_\bullet)$ 
is fixed, we obtain a homotopy equivalence $E(x,y_\bullet) \simeq \prod_{j}\Omega N$.
Consequently, we have shown

\begin{bigadd} \label{bigadd:highly-connected} Assume $p \le 1 + \Sigma + \min(a,b_1,\dots,b_j)$. 
Then the obstruction group appearing in Theorem \ref{main-thm} is isomorphic
to the direct sum of framed bordism groups
\[
\bigoplus_{(j-1)!}\Omega^{\text{\rm fr}}_{p-2-\Sigma}(\textstyle \prod\limits_j  \Omega N)\, .
\]
\end{bigadd}

\begin{ex} Suppose $P = S^p$ and $Q_i = S^{q_i}$ are spheres. Then 
$a = p-1$ and $b_i = q_i -1$. Consequently, the inequality appearing
in Addendum \ref{bigadd:highly-connected} becomes $p \le \Sigma + \mu - j$.
\end{ex}

\begin{ex} Suppose $p = 2 + \Sigma$ and $a,b_i \ge 1$.  Then 
the obstruction group of Addendum \ref{bigadd:highly-connected} is 
isomorphic to $\oplus_{(j-1)!}\Bbb Z[\pi]^{\otimes j}$, with $\pi = \pi_1(N)$.
\end{ex}

\subsection{The solution space} The space of lifts 
solving the multi-relative intersection problem
\eqref{eqn:lifting-problem} is defined by converting the vertical map appearing in that diagram
into a fibration and then taking the space of sections of this fibration
along $P$. The space of such lifts is called the {\it solution space} and is denoted by
$\cal L(f)$.

For a spectrum $E$ we let $\Omega^\infty E$ be the associated infinite loop space.

\begin{bigthm}\label{main:solution-space}
Assume that in the solution space $\cal L(f)$ is non-empty
and is equipped with a choice of basepoint. Then 
there is 
a $(1-p+\mu + \Sigma)$-connected  map
\[
\cal L(f) \to \prod_{(j-1)!} \Omega^{\infty} E(P,Q_\bullet)^{\xi+(1-2j)\epsilon}\, .
\]
\end{bigthm}

\subsection{Families of embeddings}
A variant of the multi-relative intersection 
problem involves
families of smooth
embeddings. In this instance one is given a map of $(j+1)$-ads 
$f\: P \times \Delta^{j-1}\to N$ which is also a $(j-1)$-parameter family of smooth embeddings
from $P$ to $N$. The solution of the problem in this case is to find a deformation
of ad-maps, this time through an isotopy, to 
a $(j-1)$-parameter family of embeddings having image disjoint from $Q_{J}$. 

By combining Theorem \ref{main-thm} 
with \cite[thm.~E]{GK2}, we
obtain

\begin{bigthm}[Multiple Disjunction] \label{embedding-result}  Assume $p,q_i \le n-3$ and $p \le 1 + \min(n-p-2,\mu)+\Sigma$.
Then $\chi(f) =0$  if and only if the multi-relative intersection problem 
of embeddings has a solution.
\end{bigthm}

\subsection{The embedding tower} For  a smooth manifold $P$
of dimension $p$ without boundary and a smooth manifold $N$ of dimension $n$, possibly with boundary, 
let $E(P,N)$ denote the space of smooth embeddings. When $P$ is closed, Weiss \cite{Weiss}
exhibits a tower of fibrations
 \[
 \cdots \to E_2(P,N) \to E_1(P,N)
 \]
 and compatible maps $E(P,N) \to E_k(P,N)$. 
 Up to homotopy, the $j$-th layer of 
 the tower is given by the space of compactly supported
 global sections of a certain fibration over the configuration space 
 $\tbinom{P}{j}$, the latter given by the space of subsets of $P$ 
 having cardinality $j$.  The space $E_j(P,N)$ is in some sense the best approximation to $E(P,N)$ obtained from spaces of embeddings
 $E(U,N)$ as $U$ ranges throughout the open subsets of $P$ that are diffeomorphic to a disjoint union of at most $j$ open balls.
  In what follows, we assume that $P$ is compact.
 
If $p\le n-1$, then
 $E_1(P,N)$ has the homotopy type of the space of immersions of $P$ in $N$. If $p\le n-3$, then  
the map
 \[
E(P,N)\to \lim_{j\to \infty} E_j(P,N)
 \]
 is a homotopy equivalence \cite{GW},\cite{GK2}. 
 The above motivates the following question: given a point of some stage of the tower, say $E_{j-1}(P,N)$, what are the obstructions to lifting the given point to the embedding space? 
 If $j = 2$, the work of Haefliger \cite{Haefliger}, Dax \cite{Dax}, Salomonsen \cite{Salo} and Hatcher-Quinn \cite{H-Q} provide answers to this question in the metastable range (for the discussion of this case in the context of the tower, see \cite[\S4]{Weiss}). 
 
It will be convenient to consider the following
modification of this problem. 
Fix a basepoint of $E_1(P,N)$, i.e., an immersion. Let
 $\bar E_j(P,N)$ be the fiber of  $E_j(P,N) \to E_1(P,N)$.
 Then the tower 
 \[
\cdots  \to \bar E_2(P,N) \to \bar E_1(P,N) = \ast
 \]
converges to $\bar E(P,N) = \text{fiber}(E(P,N) \to E_1(P,N))$. 
Furthermore, the layers of this tower for $j > 1$ coincide with the layers of the embedding tower.
 
Recall that $J = \{1,\dots,j\}$.
In \S\ref{sec:speculation}
we construct a spectrum with $\Sigma_j$-action
 $\cal C_J$ over the configuration space $E_J(P) := E(J,P)$, which
depends only  on the data $P$, $N$ and $j$. Let $\tau$
be  the tangent bundle of $E_J(P)$ 
(i.e., restriction of the cartesian product 
$j$-copies of the tangent bundle of $P$). Then we can twist
$\cal C_J$ by $-\tau$ to obtain a fiberwise spectrum with
$\Sigma_j$-action
${}^{-\tau}\cal C_J$ over $E_J(P)$. In particular, one can speak
about the equivariant homology of $E_J(P)$ with coefficients
in ${}^{-\tau}\cal C_J$. 

We will define an invariant
 \[
 \mu\: \pi_0(\bar E_{j-1}(P,N)) \to 
 H^{\Sigma_j}_0(E_J(P);{}^{-\tau}\cal C_J)
 \] 
which vanishes on the image of $\pi_0(\bar E_{j}(P,N))$.

 \begin{bigthm} \label{bigthm:embedding-sequence} Assume $j \ge 2$ and $N$ is $r$-connected with $r\le n-2$. Assume additionally
 \[
 r  \, \ge\,  p-1 - (j-1)(n-p-2)\, .
 \]
If $x\in \bar E_{j-1}(P,N)$, then $\mu(x)=0$ implies
that $x$ lifts to $\bar E_j(P,N)$.
\end{bigthm}

If $N$ is contractible then we can take $r=n-2$. In this case
the displayed inequality $ r  \ge p-1 - (j-1)(n-p-2)$ is automatically satisfied:

\begin{bigcor} \label{bigcor:embedding-sequence} Assume $j \ge 2$ and that $N$ is contractible. If
$\mu(x)=0$, then $x \in \bar E_{j-1}(P,N)$ lifts to $\bar E_j(P,N)$.

\end{bigcor}

\begin{rem} By \cite{GW}, the map $\bar E(P,N) \to
\bar E_j(P,N)$ is $((j+1)(n-p-2) + 3-n)$ connected. Consequently,
in both Theorem \ref{bigthm:embedding-sequence} and Corollary \ref{bigcor:embedding-sequence}
if $\mu(x) = 0$ then $x\in \bar E_{j-1}(P,N)$ will lift to $\bar E(P,N)$ if in addition $(j+1)(n-p-2) + 3-n \ge 0$.
\end{rem}

\subsection{Link maps} 
Our main results can also be used to study higher order linking phenomena.
Given connected closed manifolds $P_1,\dots,P_j$ and a connected manifold $N$, 
a ($j$-component) {\it link map} is a continuous function
\[
f\: P_1 \amalg \cdots \amalg P_j \to N
\]
such that $f(P_i) \cap f(P_k) = \emptyset$ for $i \ne k$. The space of link maps will be denoted by $\cal L(\bold P,N)$.\footnote{The path components of $\cal L(\bold P,N)$ are called {\it link homotopy classes}. The latter is usually studied in the special case when $N = \Bbb R^n$ and the $P_i$ are spheres \cite{Milnor},
\cite{Massey}, \cite{Koschorke}.}
Fix an embedding
$J \to  N$, where we recall again that  $J = \{1,2,\dots, j\}$.
We will also identify $J$ with its image in $N$.

We define the {\it trivial
link map} to be the link map given by sending the component $P_i$ to $i \in J$, i.e., the trivial link map factors as the composition
$
P_1 \amalg \cdots \amalg P_j \to J \subset N
$,
where the first map is the canonical surjection 
from a space onto its set of components.
The trivial link map equips $\cal L(\bold P,N)$ with a basepoint. A link map is {\it trivializable}
if it admits a path to the trivial link map in the space of link maps.

\begin{defn} The space of
 {\it (homotopy coherent) Brunnian link maps} 
 \[
 \cal B(\bold P,N)
 \]
is the total homotopy fiber of 
the $j$-cube of based spaces
\[
S \mapsto \cal L^S(\bold P,N)
\]
where $\cal L^S(\bold P,N)$ is the space
of maps $f\: P_1 \amalg \cdots \amalg P_j \to N$ such that
for every $S\subset J$ the restriction 
\[
f_S := f_{|P_S}\: \textstyle\coprod_{i \in S} P_i \to N
\]
is an $|S|$-component link map.
\end{defn}

Since $\cal B(\bold P,N)$ is the homotopy fiber
of  the map 
\[
\cal L^J(\bold P,N) \to \holim_{S\subsetneq J}\cal L^J(\bold P,N)\, ,
\]
a point of $\cal B(\bold P,N)$ determines
a link map $f\in\cal L^J(\bold P,N)$ with the property that
any proper sub-link map is trivializable. In particular,  $f$ 
satisfies the classical Brunnian condition \cite{Milnor},\cite{Debrunner}.

Restricting now to the case when $N=\Bbb R^n$, 
we will construct
in  \S\ref{sec:linking} a {\it higher stable linking number map}\footnote{For link maps
of circles in three dimensional euclidean space, 
it seems likely that on path components, our map
coincides with Milnor's $\mu$-invariants \cite{Milnor}.}
\begin{equation} \label{eqn:lambda-intro}
 \lambda\:\cal B(\bold P,\Bbb R^n) @>>> \prod_{i=1}^{(j-2)!}
 F^{\text{st}}(\textstyle\prod_{i=1}^j P_i,S^{(j-1)(n-2)+1})\, ,
\end{equation}
where for an unbased space $X$ and a spectrum $E$,
$F^{\text{st}}(X,E)$ denotes the function space of stable maps
from $X$ to $E$ i.e., the function space $F(X,\Omega^\infty E)$.

A result of  
Goodwillie and Munson in the case $j=2$ \cite[th.~1.1]{GM}, suggests
to us the following:

\begin{bigconj} \label{bigconj:lambda} 
The map $\lambda$ is $(1+\Sigma')$-connected,
where
\[
\Sigma' = \sum_{i=1}^j (n-2p_i-2)\, .
\]
\end{bigconj}

\noindent (For variant forms of this statement see \S\ref{sec:linking}.)
We submit the following evidence for  Conjecture \ref{bigconj:lambda}:

\begin{bigthm}[Realization of higher linking numbers] \label{bigthm:realization} Assume that
$P_i$ embeds in $\Bbb R^n$ and $n-p_i\ge 2$ for $2\le i \le j$. 
Then the higher stable linking number map $\lambda$
induces a surjection on homotopy groups in degrees $\le 1 - \hat p + \Sigma$, where
\[
\hat p := \max_{2\le i\le j} p_i \quad \text{\rm and } \quad
\Sigma = \sum_{i=1}^j (n-p_i-2)\, .
\]
\end{bigthm}

In the above, we do not need to assume that the embeddings are pairwise disjoint.
Since $1 - \hat p + \Sigma \ge 1 + \Sigma'$, it follows that
$\lambda$ induces a surjection on homotopy groups in degrees
$\le 1 +\Sigma'$. Hence, Theorem \ref{bigthm:realization}
gives evidence for the validity of Conjecture \ref{bigconj:lambda}. 
Further evidence is contained in \S\ref{sec:linking}. 
Our results on link maps overlap with those of Munson 
\cite{Munson}. Our methods are homotopy theoretical, whereas
Munson relies on bordism and transversality. It seems 
likely to us that Theorem \ref{bigthm:realization} could  also be
extracted from  Munson's approach, possibly at the
expense of a dimension.

\begin{out} Section \ref{sec:setup} is a breezy exposition on the basic
definitions as well as the machinery 
used throughout the paper.
 Section \ref{sec:cocart} is about strongly cocartesian cubes of spaces, and the main technical results of the paper are stated there. 
Section \ref{sec:geometry} recasts the results of section \ref{sec:cocart} in the setting of
homotopical intersection theory to give a proof of Theorems \ref{main-thm}
and \ref{main:solution-space} modulo the proof of Theorem
 \ref{thm:also-hard}.  
In section \ref{sec:proofs} we prove Theorem \ref{thm:also-hard} which is one of our main technical results. In section \ref{sec:disj} we combine Theorem \ref{main-thm} with 
 \cite[thm.~E]{GK2} to obtain a multiple disjunction result for smooth embeddings. Section ref{sec:speculation} contains
 the proof of Theorem \ref{bigthm:embedding-sequence}. In section \ref{sec:linking} 
we apply our machinery to the study of spaces of
link maps.  
 \end{out}

\begin{acks} Bruce Williams passed away on January 11, 2018 before
the final revision of this paper was completed.  Bruce was a close friend and an inspiring mentor to the first author. 
 
We wish to thank the referee doing a thorough job.
The referee's suggestions entailed significant changes in both
 the exposition and the proofs
of many of the results. The first author is convinced that the payoff was worthwhile and he thinks that Bruce Williams would have agreed. The referee suggested to make the common theme of the different parts clearer. We hope that this version of the paper reflects the referee's suggestion. Perhaps a slogan like
``manifold applications of the higher Blakers-Massey theorem'' 
could serve as the leitmotif (pun intended).

 We are much indebted to Tom Goodwillie for explaining to us
 the idea of the proofs of  Lemmas \ref{lem:wedge-cube} and \ref{lem:GK-improved}. We 
 thank Brian Munson for explaining to us aspects of his paper 
 \cite{Munson} and we also thank Uwe Kaiser for helping us with the link homotopy literature. 
The initial version of the paper was partially written while the first author
visited the University of Copenhagen. He is indebted to Lars Hesselholt for providing him with support from the Bohr Professorship 
to finance his stay. The first author was supported by
Collaboration Grant 317496 from the Simons Foundation.
\end{acks}

\section{Language \label{sec:setup}} 
\subsection{Spaces}
Let $\cal T$ be the category of compactly generated spaces. Then $\cal T$ is a
Quillen model category in which the weak equivalences are the weak homotopy equivalences,
the fibrations are the Serre fibrations and the cofibrations are the retracts 
of relative cell complexes \cite[ch.~2, \S3]{Quillen} 
(a relative cell complex is a pair of spaces $(Y,A)$
such that $Y$ is obtained  from $A$ by attaching cells).
A space $X$ is $r$-{\it connected} if every map $S^k \to X$
for $k \le r$ is homotopic to a constant map; here $S^k$ is the sphere of dimension $k$. In particular,
the empty space is $(-2)$-connected and every non-empty
space is (at least) $(-1)$-connected.
A map $f\: X \to Y$ is 
$r$-connected if its homotopy fiber at any basepoint is $(r-1)$-connected.
An $\infty$-connected map is, by definition, a weak equivalence.

A commutative square of spaces
\begin{equation} \label{eqn:first-square}
 \xymatrix{
 A \ar[r] \ar[d] & C \ar[d] \\
B \ar[r] & D
 }
 \end{equation}
 is {\it $r$-cocartesian}  if the map 
 \[
 \hocolim (B\leftarrow A \to C) \to D
 \]
 is $r$-connected. 
  
 Dually, the square \eqref{eqn:first-square} is {\it $r$-cartesian}
 if the map
 \[
 A\to \hocolim (B\to D \leftarrow C) 
 \]
 is $r$-connected.

 \begin{defn} Let
  \begin{equation} \label{eqn:fibersequence}
  X\to Y \to Z
  \end{equation}
  be maps of spaces equipped  with a homotopy to a constant $z$.
One says that \eqref{eqn:fibersequence} is  a
{\it homotopy fiber sequence in degrees $\le s$}  
 if the induced map from $X$ to the homotopy fiber of $Y\to Z$ is $s$-connected. If  this condition holds for all integers $s$, then \eqref{eqn:fibersequence}
 is called a {\it homotopy fiber sequence}.
 
Dually, if the induced map from the homotopy cofiber of $X\to Y$ to 
 $Z$ is $s$-connected, then one says that \eqref{eqn:fibersequence} is a {\it homotopy cofiber sequence in degrees $\le s$} (respectively, a {\it homotopy cofiber sequence} if the condition holds for all $s$). 
  \end{defn}
 
When the square \eqref{eqn:first-square} is $\infty$-cocartesian and $C$ is contractible, then $A\to B \to D$ is a homotopy cofiber sequence once a contraction 
 $C \times [0,1] \to C$  
is specified. The dual case is analogous.

\subsection{Fiberwise spaces}
For an object $X\in \cal T$, we let $\cal T(X)$ denote the category of spaces over $X$.
This is the category whose objects are pairs $(Y,r)$ such that
$r\: Y \to X$ is a map. A morphism $(Y,r) \to (Y',r')$ is a map 
$f\: Y \to Y'$ such that $r'\circ f = r$. We more often than not suppress
the structure map $r\:Y \to X$ when specifying an object and write
$Y$ in place of $(Y,r)$. 

Similarly, let $\cal R(X)$ denote the
category of retractive spaces over $X$. This has objects $(Y,r,s)$
where $r\: Y \to X$ and $s\: X \to Y$ are maps such that $r\circ s$ is the identity map. A morphism $(Y,r,s) \to (Y',r',s')$ is a map $f\: Y \to Y'$ such that
$r'\circ f = r$ and $f\circ s =s'$.  Again, the structure maps are usually surpressed.

Note that the case $\cal R(*)$ gives the category of based spaces.
We sometimes regard objects of $\cal R(X)$ as objects of $\cal T(X)$ by means
of the forgetful functor. When $X = \ast$ we usually write 
$\cal T_\ast$ in place of $\cal R(\ast)$, i.e., the category of based spaces.

Both $\cal T(X)$ and $\cal R(X)$ have simplicial model category structures
where a weak equivalence (cofibration, fibration) in each case is a morphism whose
underlying map of spaces is a weak homotopy equivalence 
(cofibration, fibration) of spaces \cite[2.8, prop.\ 6]{Quillen}. In particular,
 the set of (fiberwise) homotopy classes $[Y,Z]_{\cal T(X)}$ is defined for objects
$Y,Z$ of $\cal T(X)$. Similarly, one can define homotopy classes in $\cal R(X)$. 
If $Y \in \cal T(X)$ is an object, let $Y^+\in \cal R(X)$ be the object given by
$Y \amalg X$ with evident structure maps. If $Z \in \cal R(X)$ is an object, then we
have $[Y^+,Z]_{\cal R(X)} = [Y,Z]_{\cal T(X)}$.
As usual, when defining homotopy classes $[Y,Z]_{\cal T(X)}$,
$Y$ is replaced by a cofibrant approximation and $Z$
is replaced by a fibrant approximation.  

 A morphism $Y \to Z$ in either $\cal T(X)$ or $\cal R(X)$ is said
 to be {\it $j$-connected} if and only if its underlying map in $\cal T$ is $j$-connected.  An object $Y$ is said to be $j$-connected
 if and only iff the structure map $Y \to X$ is $(j+1)$-connected.
 A commutative square in $\cal T(X)$ or $\cal R(X)$ is {\it $j$-cocartesian}
 ($j$-cartesian) if it is so when considered in $\cal T$ (here $j$ may be  
 $\infty$).

We say an object $Y$ of $\cal T(X)$ or $\cal R(X)$
 has dimension $\le s$ it is built up from the initial object by attaching cells of dimension at most $s$. In $\cal T(X)$ this means
 that the underlying space of $Y$
is a cell complex of dimension at most $s$. In $\cal R(X)$ it means
that the pair $(Y,X)$ is a relative cell complex of dimension at most $s$. In either case we write $\dim Y \le s$.

A sequence of maps  $A \to Y \to C$ in $\cal T(X)$ 
forms a {\it homotopy cofiber sequence} (respectively, in degrees $\le r$) if it comes equipped with a homotopy from 
 $A\to C$ to a composition of the form $A\to X \to C$ 
 (where $X$ is viewed as the terminal object) 
 such that the induced map
 from the homotopy cofiber of $A\to Y$ (i.e.,  the homotopy colimit
 of $X \leftarrow A\to Y$) to $C$
 is a weak equivalence (respectively $r$-connected).  The dual notion of homotopy fiber sequence (in degrees $\le r$) is defined analogously.

\begin{lem} \label{lem:excision} Suppose that $A \to Y \to C$ is a homotopy cofiber sequence
of $\cal T(X)$. Assume that $A$ is $r_1$-connected and $C$ is $r_2$-connected.
Then $A \to Y \to C$ is a homotopy fiber sequence in dimensions $\le r_1+r_2$.
\end{lem}

\begin{proof} The square
\[
\xymatrix{
A \ar[r] \ar[d] & Y \ar[d]\\
X \ar[r] & C 
}
\]
has a preferred commuting homotopy making it $\infty$-cocartesian.
The result follows from Blakers-Massey theorem 
\cite[thm.~4.23]{Hatcher}, \cite[p.~309]{Good}.
\end{proof}

\begin{cor}\label{cor:excision}  Assume in addition that  $Z\in \cal T(X)$ is an object of dimension $\le r_1 +r_2$. 
Then the sequence of sets
\[
[Z,A]_{\cal T(X)} \to [Z,Y]_{\cal T(X)} \to [Z,C]_{\cal T(X)}
\]
is exact.
\end{cor}

\noindent (Explanation: the set $[Z,C]_{\cal T(X)}$ has a preferred basepoint
given by $Z \to X' \to C$. Any element of $[Z,Y]_{\cal T(X)}$ which maps
to the basepoint lifts back to $[Z,A]_{\cal T(X)}$.)
\medskip

\subsection{Fiberwise suspension}
The {\it unreduced fiberwise suspension} of an object $Y \in \cal T(X)$
is the object of $\cal R(X)$ given by the double mapping cylinder
\[
S_X Y := (X \times 0) \, \cup \, Y \times [0,1]\,  \cup (X \times 1)\, ,
\]
where the structure map $S_X Y \to X$ is obvious and the 
structure map $X \to S_X Y$ is given by $X\times 0$.
This gives a functor $S_X\: \cal T(X) \to \cal R(X)$.
Similarly, $\cal R(X)$ has a {\it reduced fiberwise suspension} functor
$\Sigma_X \: \cal R(X) \to \cal R(X)$ defined as follows: given 
an object $Y \in \cal R(X)$, we take $\Sigma_X Y$  to be the 
pushout of the diagram $X \leftarrow S_X X \to S_X Y$.  
If $Y$ is cofibrant, then the map $S_X Y \to \Sigma_X Y$
is a weak equivalence. The functor $\Sigma_X$ has a 
right adjoint $\Omega_X$, called the 
{\it fiberwise loop functor}.

Given objects $Y,Z \in \cal R(X)$ define
\[
\{Y,Z\}_{\cal R(X)} := \colim_k [\Sigma^k_X Y, \Sigma^k_X Z] \, .
\]
This is the abelian group of fiberwise stable homotopy classes from
$Y$ to $Z$.

\subsection{Fiberwise smash product}
Given objects $Y,Z \in \cal T(X)$, we have the fiber product $Y \times_X Z \in \cal T(X)$
which is defined as the limit of the diagram $Y \to X \leftarrow Z$. If $Y,Z \in \cal R(X)$, the
fiberwise wedge (or coproduct) $Y \vee_X Z$ is the object of $\cal R(X)$ given by 
the pushout of the inclusions $Y \supset X \subset Z$. The {\it (internal fiberwise) smash product}
is the object $Y\smsh_X Z$ given by the pushout of the diagram 
$X \leftarrow Y \vee_X Z \subset Y \times_X Z$. As is usual with 
most functors in the model category theoretic setting, this construction needs to be suitably derived to have a meaningful homotopy type (in this instance $Y$ and $Z$
should be made fibrant and cofibrant). 
To avoid notational clutter, we will be intentionally 
sloppy: we will write the underived smash product but the reader should understand that it
needs to be derived to have a sensible homotopy theoretic meaning.

\subsection{Fiberwise Thom spaces} 
Given an object $Y \in \cal T(X)$ and an inner product bundle $\xi$ over $Y$,
the {\it fiberwise Thom space} is the object of $\cal R(X)$ given by
\[
T_X(\xi) = D(\xi) \cup_{S(\xi)} X \, .
\]
By collapsing $X$ to a point we obtain usual Thom space $X^\xi:= D(\xi)/S(\xi)$,
which in the present notation appears as $T_*(\xi)$.

Let $\eta$ be an inner product bundle over another object $Z\in \cal T(X)$.  
Let $p\: Y \times_X Z\to Y$
and $q\:Y \times_X Z \to Z$ be the projections.
Then the Whitney sum $p^*\xi \oplus q^*\eta$ is an inner product bundle over $Y\times_X Z$.
The following is just an unravelling of definitions (and is well-known when $X$ is a point).

\begin{lem} \label{Whitney-sum-formula} 
There is a preferred isomorphism of $\cal R(X)$
\[
T_X(p^*\xi \oplus q^*\eta) \,\, \cong \,\, 
T_X(\xi) \smsh_X T_X(\eta)\, .
\]
\end{lem}

\subsection{Fiberwise spectra}
Using $\Sigma_X$ also enables one to define spectra built from
objects of $\cal R(X)$.  A {\it fiberwise spectrum} $\cal E$ is
a collection of objects $\cal E_n \in \cal R(X)$ for $n = 0,1,\dots$ together
with morphisms $\Sigma_X \cal E_n \to \cal E_{n+1}$.  
Note that $\cal E$ comes equipped with a {\it zero section}, namely,
the collection of structure maps $X\to \cal E_n$ for $n\ge 0$.
A morphism
of fiberwise spectra is the evident thing.

If $\cal E$ is a fiberwise spectrum
then the associated fiberwise infinite loop space $\Omega_X^\infty \cal E$ is an object of $\cal R(X)$. Fiberwise spectra
form a model category (see e.g., \cite{Schwede}; for a more
detailed treatment see \cite{MS}). 

Here are two examples:

\begin{ex}[Trivial fiberwise spectra]
Start with an ordinary spectrum $E$ given by based spaces 
$\{E_n\}_{n \ge 0}$ and structure maps $\Sigma E_n \to E_{n+1}$.
Form $E_n \times X$ for $n \ge 0$. These fit into 
a fiberwise spectrum $E \times X$, where the structure map
$\Sigma_X (E_n \times X) \to E_{n+1} \times X$ is given by 
noticing that $\Sigma_X (E_n \times X) \cong (\Sigma E_n) \times X$.
\end{ex}

\begin{ex}[Fiberwise suspension spectra]
 Start with any object $Y \in \cal  R(X)$ and form
the iterates $\Sigma_X^n Y$. These give a fiberwise spectrum
$\Sigma^\infty_X Y$, using the identity maps for the structure maps.
\end{ex}

We remark that
the zero section of $\cal E$ gives a morphism
$\Sigma_X^\infty X^+\to \cal E$.

Given an object $Z\in {\cal R(X)}$ and a fiberwise spectrum
$\cal E$ we define
\[
\{Z,\cal E\}_{\cal R(X)} := \colim_{n} [Z,\Omega_X^\infty\cal E]_{\cal R(X)} \, .
\]
For example, if $\cal E = \Sigma^\infty_X Y$ is a fiberwise suspension spectrum, then
 $ \{Z,\cal E\}_{\cal R(X)} =  \{Z,Y\}_{\cal R(X)}$.

\subsection{Homology and cohomology}
Let $\cal E$ be a fiberwise spectrum over $X$ (which we take to be fibrant).
Then an object $Z \in \cal T(X)$ (which we take
to be cofibrant) with structure map
$p\: Z \to X$ gives rise to a fiberwise spectrum over $Z$
\[
p^*\cal E
\]
whose $k$-th space is the pullback of $\cal E_k \to X$ along $p$.
Let $(p^*\cal E)^\flat$ denote the effect of making
$p^*\cal E$ cofibrant. Then for each $n \ge 0$ we have
a cofibration $Z \to (p^*\cal E)^\flat_n$ and as $n$ varies the quotient
spaces $(p^*\cal E)^\flat_n/Z$ form a spectrum denoted $H_\bullet(Z;\cal E)$.  
The {\it homology groups} of $Z$ with coefficients in $\cal E$ are the homotopy
groups of this spectrum.

To define cohomology
we take, for each $n$,  the
space of sections of $\cal E_n \to X$ along the map $Z \to X$ (this is
the same thing as the space of maps $Z \to \cal E_n$ which commute with
the structure map to $X$.
As $n$-varies, these spaces form a spectrum $H^\bullet(Z;\cal E)$.
The {\it cohomology groups} of $Z$ with coefficients in $\cal E$ are defined to be
 homotopy groups of this spectrum, i.e.,  
 \[
 H^i(Z;\cal E) = \{Z^+,\Sigma_X^{i}\cal E\}_{\cal R(X)}\, .
 \]

\subsection{Induction and restriction \label{subsec:induction}}
Let $f\: X \to Y$ be a map of spaces. Then 
a fiberwise spectrum $\cal E$ over $Y$ gives rise to a fiberwise
spectrum $f^*\cal E$ over $X$ by taking base change. 
This operation defines a {\it restriction} functor
from fiberwise spectra over $Y$ to fiberwise 
spectra over $X$ (the  construction is homotopy invariant when $\cal E$ is fibrant).
Using $f$ to regard $X$ as an object of $\cal T(Y)$, we obtain a tautological identification
$H^\bullet(X;\cal E) = H^\bullet(X,f^*\cal E)$, where on the right side
$X$ is viewed as an object of $\cal R(X)$ using the identity. 

Suppose $\cal F$ is a fiberwise spectrum over $X$. Then we obtain a fiberwise
{\it pushforward} spectrum
over $Y$, denoted $f_*\cal F$ in which $(f_*\cal F)_k = (\cal F_k) \cup_f Y$
(the construction is homotopy invariant when $\cal F$ is cofibrant). 
The operation $E \mapsto f_\ast E$ is also called {\it induction.}  
Note that $H_\bullet(X;\cal F) = H_\bullet(Y;f_*\cal F)$ tautologically. 
Note that $(f_\ast,f^\ast)$ is an adjoint pair.

\subsection{Poincar\'e duality} Let $\xi$ be a finite dimensional vector bundle over $X$.
Let $S^{\xi}$ denote the fiberwise one-point compactification of $\xi$.
Then $S^{\xi}$ is an object of $\cal R(X)$. More generally, if $\xi$
is a virtual bundle, i.e.,
 $\xi +\epsilon^j$ is identified with
a finite dimensional vector bundle $\eta$ for some $j$, then we define 
$S^{\xi}$ is this case to be a fiberwise spectrum over $X$ given by 
the $j$-fold desuspension of $S^{\eta}$. 

Given a fiberwise spectrum $\cal E$ over $X$, set 
\[
{}^{\xi} \!\cal E := S^\xi \smsh_X \cal E \, .
\]
When $\xi$ is a vector bundle then the definition of the right side is
given by the fiberwise smash products in each degree, i.e., $S^\xi \smsh_X \cal E_k$.
In the virtual bundle case one merely fiberwise desuspends $S^\eta\smsh_X \cal E$ $j$-times.

\begin{thm}[Poincar\'e duality {\cite{Klein-dualizing}, \cite[th.~6.2]{KW2}, \cite[thm.~19.6.1]{MS}}] 
\label{thm:duality}  Suppose $f\: P \to X$ is a map
in which $P$
is a closed smooth manifold of dimension $d$.
Let $-\tau_{P}$ be the virtual stable normal bundle 
given by the negation of the tangent bundle of $P$.
Then for any fiberwise spectrum $\cal E$ over $X$ there is a preferred
weak equivalence of spectra
\[
H^\bullet(P;\cal E) \,\, \simeq \,\, H_\bullet(P;{}^{-\tau_P}\!f^*\cal E)\, .
\]
\end{thm}

\begin{rem} \label{rem:compact-supports} More generally, if
 $P$ is an open manifold then there is
a  weak equivalence 
\[
H^\bullet_{\text{cs}}(P;\cal E) \,\, \simeq \,\, H_\bullet(P;{}^{-\tau_P}\!f^*\cal E)\, ,
\]
where the left side denotes cohomology with compact supports. The latter is defined  by taking the spectrum of sections of 
$\cal E$ which
coincide with the  
zero section near infinity.
\end{rem}

\section{Strongly cocartesian cubes \label{sec:cocart}} 
\subsection{Cubical diagrams} 
For a finite set $J$, we let $2^J$ be the poset of consisting of the subsets of $J$
partially ordered by inclusion.
A {\it $J$-cube} in a category $\cal C$ 
is a contravariant functor 
\[
A_\bullet\: 2^{J} \to \cal C,\qquad S \mapsto A_S\, .
\]
(If $J$ has cardinality $j$, we
also say that $A_\bullet$ is a {\it $j$-cube}.)
Since $A_\bullet$ is contravariant, the initial vertex
is $A_J$ and the terminal vertex is $A_\emptyset$.
When $J = \{i\}$ we usually write $A_S = A_i$. 

In what follows
we will only consider $J$-cubes in which
the target category $\cal C$ is either 
$\cal T(X)$ or $\cal R(X)$ for some space $X$,  and often enough, we
shall be interested in  the case when $X$ is a point.

A {\it weak equivalence} of  $T$-cubes
$A_\bullet \to B_\bullet$ is a natural transformation  
such that $A_S \to B_S$ is a weak equivalence for
each $S$, i.e., an objectwise weak equivalence.
 Two $J$-cubes are said to be {\it weakly equivalent} if there is a finite zig-zag
of weak equivalences connecting them.

\begin{defn}[{\cite[defn.~1.3]{Good}}]
A $J$-cube $A_\bullet$ is {\it $r$-cartesian} if the map
\begin{equation}\label{eqn:holim-map}
A_J\to \holim \limits_{S\subsetneq J} A_S
\end{equation}
is $r$-connected. Similarly, $A_\bullet$ is {\it $r$-cocartesian} if the map
\begin{equation} \label{eqn:hocolim-map}
\hocolim\limits_{S\ne \emptyset}  A_S \to A_\emptyset
\end{equation}
is $r$-connected. In both cases $r$ may be $\infty$.
\end{defn}

We remark that when $A_\bullet$ is a cube in 
which the maps $A_S \to A_T$ are based
for $|S| < j$, then the target of \eqref{eqn:holim-map} inherits a basepoint. In this case, we will say that $A_\bullet$ is {\it almost based}.

\begin{defn} The {\it total homotopy cofiber} of 
$A_\bullet$  is the homotopy cofiber of the map \eqref{eqn:hocolim-map}. If $A_\bullet$ is an almost based $J$-cube, then its {\it total homotopy fiber} is the homotopy fiber of \eqref{eqn:hocolim-map}
taken at the preferred basepoint.
\end{defn}

For fixed subsets $U \subset W  \subset J$, one has a
 {\it $(W,U)$-face} of $A_\bullet$ given by restricting $A_\bullet$ to those
 $A_S$ for which $U \subset S \subset W$. This is a $(W\setminus U)$-cube and every face
 of $A_\bullet$ arises in this fashion. When $|W\setminus U| = k$ we also call this 
 a {\it $k$-face} of $A_\bullet$.

 \begin{defn}[{\cite[defn.~2.1]{Good}}] \label{defn:strongly-cocart} A $J$-cube
 $A_\bullet$ is  {\it strongly cocartesian}
 if each 2-face of $A_\bullet$  is $\infty$-cocartesian. 
\end{defn} 

In Definition \ref{defn:strongly-cocart}, it is enough to check 
the condition on each $2$-face meeting the initial vertex $A_J$
(i.e., those $(W,U)$-faces in which $|W\setminus U|=2$ and $W=J$
cf.\ \cite[{\it loc.\ cit.\rm}]{Good}). 

Henceforth, we set 
\[
J := \{1,2,\dots,j\}\, .
\]

\begin{ex}[Wedge Cubes]  Let $X_1,\dots,X_j$ be cofibrant based spaces. For $T \subset J$,
let $A_T$ be
the wedge $\vee_{i\in T} X_i$ (by convention $A_\emptyset$ is a point). 
This defines a
strongly cocartesian $j$-cube $A_\bullet$
whose maps are given by
projections onto summands.

More generally, let $X_1,\dots, X_j \in \cal R(X)$ be cofibrant. Let $A_T$ be the fiberwise wedge of $X_i$ as $i$ varies in $T$.
Then $A_\bullet$ is strongly cocartesian.
\end{ex}

\begin{ex}[Backwards Wedge Cubes] With $X_1,\dots,X_j \in \cal R(X)$ as above, let 
$B_T$ be the fiberwise wedge of those $X_i$ with $i \in J \setminus T$. 
The maps of this cube are
inclusions of summands. Then $B_\bullet$ is strongly cocartesian.
\end{ex}

\begin{ex}[Suspension] Let $A_\bullet$ be a strongly cocartesian $j$-cube
of $\cal T(X)$. 
Then the $j$-cube $S_X A_\bullet$
given by $T \mapsto S_X A_T$ 
is also strongly cocartesian.  Similarly, if $A_\bullet$ is a strongly cocartesian $j$-cube
of $\cal R(X)$, then the cube of reduced fiberwise suspensions
$\Sigma_X A_\bullet$ is strongly cocartesian.
\end{ex}

\begin{lem} \label{lem:wedge-cube} Let  $A_\bullet$ be a strongly cocartesian $j$-cube  
of connected based spaces in which $A_\emptyset$ is a point. 
Then the suspended $j$-cube  $\Sigma A_\bullet$ is weakly equivalent to a wedge cube
$B_\bullet$ in which $B_i = \Sigma A_i$  for $i\in J$. 
\end{lem}

\begin{proof} 
The following sketch was provided to us
by Tom Goodwillie.  
Let $B_T$ be the wedge  of $\Sigma A_i$
for all $i \in T$, but write this as the wedge, over all $i\in J$, 
of either 
\begin{itemize}
\item $\Sigma A_i$ if $i\in T$, or 
\item  $\ast$ if $i\notin T$.
\end{itemize}

Define a map $\Sigma A_T \to B_T$ as follows.
First do a pinch to go from $\Sigma A_T$ to the wedge of 
$j$ copies of $\Sigma A_T$ indexed by $i \in J$.
Now map that to $B_T$ by sending the $i$-th copy of $\Sigma A_T$ to 
$\Sigma A_i$ using the original map $A_T \to A_i$ if $i\in T$, 
or the constant map to a point if $i \notin T$.

The above recipe defines a map of $j$-cubes $\Sigma A_\bullet \to B_\bullet$. By the Whitehead theorem, it 
suffices to show that the map $\Sigma A_T \to B_T$
is a homology isomorphism for all $T\subset J$. 
Let $C_T$ be the homotopy cofiber of this map. Then
$T\mapsto C_T$ is also a strongly cocartesian $j$-cube.
It is enough to show that $C_T$ has trivial reduced homology.
If $T$ is a singleton, this is clear since the maps 
$\Sigma A_i \to B_i$ are homotopic to the identity. By a straightforward induction argument, we can assume that
$C_T$ has trivial homology for $|T|\le j-1$. We are reduced to
showing that $C_J$ has trivial homology. But the homology of
$C_J$ coincides with the homology of the total homotopy cofiber
of the cube $C_\bullet$ with a degree shift by $j$. Since
$C_\bullet$  is strongly cocartesian, the total homotopy cofiber is contractible. Hence, $C_J$ has trivial homology.
\end{proof}

Given a strongly cocartesian $j$-cube $A_\bullet$, 
let $C(A_\bullet)$ denote the homotopy colimit
\begin{equation} \label{eqn:hocolim}
\hocolim (A_{\emptyset} \leftarrow A_{J} \to \holim_{S\neq J} A_S) \, .
\end{equation}
Then $C(A_{\bullet})$ is a retractive space over $A_{\emptyset}$. 
In what follows we rename 
\[
X := A_{\emptyset}\, .
\] 
Then 
$C(A_\bullet)\in \cal R(X)$  and one has  a homotopy cofiber
sequence of $\cal T(X)$
\begin{equation} \label{eqn:cofiber-sequence}
A_{J} \to \holim_{S\neq J} A_S \to C(A_\bullet)\, .
\end{equation}
\begin{notation} \label{notation-throughout}
For a sequence of integers $r_1,...,r_j$ we write
\[
\Sigma = \sum_i r_i \qquad \text{and } \qquad \mu = \min_i r_i\, .
\]
\medskip
If $1\le i \le j$ and $T \subset J$ set 
\[
T_i := T \setminus \{i\}\, .
\]
\end{notation}

\begin{hypo} \label{hypothesis-throughout} $X$ is $0$-connected. Furthermore,
for $1\le i \le j$, 
the map
\[
A_{J} \to A_{J_i}
\]
is $(r_i+1)$-connected, where $r_i \ge 0$.
\end{hypo}
 
Note that $A_T \to A_{T_i}$ is
also $(r_i+1)$-connected for all $T\subset S$, since $A_\bullet$ is strongly cocartesian. We assume \ref{hypothesis-throughout} holds
throughout the rest of this section.

\begin{prop} \label{prop:cofiber}
Let $Z\in \cal T(X)$ be an object of dimension $\le 1+ \mu + \Sigma$. Then the sequence
\[
[Z,A_{J}]_{\cal T(X)} \to [Z,\holim_{S\neq \emptyset} A_S]_{\cal T(X)} \to 
[Z,C(A_\bullet)]_{\cal T(X)}
\]
is exact.
\end{prop}

\begin{rem} Note that $[Z,C(A_\bullet)]_{\cal T(X)}$ is a pointed set. As in 
Corollary \ref{cor:excision}, exactness means that an element of $[Z,\holim_{S\neq \emptyset} A_S]_{\cal T(X)}$ pushes forward to the basepoint if and only if it lifts to an element
of $[Z,A_{J}]_{\cal T(X)}$.
\end{rem}

\begin{proof} The object $A_J \in {\cal T}(X)$ is 
$\mu$-connected. 
The higher Blakers-Massey  theorem for cubical diagrams
\cite[thm.~2.5]{Good} (or cf.~\cite[thm~2.3]{Good3}), 
says that $A_\bullet$ is $(1+\Sigma)$-cartesian,
Consequently, $C(A_\bullet)\in {\cal T}(X)$ is a $(1+\Sigma)$-connected object.
The conclusion now follows from Corollary \ref{cor:excision}.
\end{proof}

\subsection{Identification of $C(A_\bullet)$}
In the remainder of this section we identify $C(A_\bullet)$ 
up through dimension $1+\mu+\Sigma$.

Let 
\begin{equation} \label{eqn:Wj}
W_j := \bigvee_{(j-1)!} S^{2-2j}
\end{equation}
be  the wedge
of $(j-1)!$-copies of the $(2-2j)$-sphere spectrum.

Let 
\[
\cal W_j = X \times W_j
\]
 be the trivial fiberwise spectrum on $W_j$. 
 
 \begin{thm} \label{thm:also-hard} With respect to the above
 assumptions, there is a
preferred map 
\begin{equation} \label{eqn:weak-equivalence-spectra}
C(A_\bullet) \to \Omega^{\infty} (\cal W_j \smsh_X (\underset {i\in J\quad }{\smsh_X} S_X  A_i))\, , 
\end{equation}
which is $(2+\mu+\Sigma)$-connected.
\end{thm}
 
\noindent The proof of Theorem \ref{thm:also-hard} is deferred to \S\ref{sec:proofs}.
If we combine Theorem \ref{thm:also-hard} with Proposition \ref{prop:cofiber} we
obtain

 \begin{cor} \label{useful} Let $Z\in \cal T(X)$ be an object such that 
 $\dim Z \le 1+\mu +\Sigma$.
 Then there is an exact sequence
 \[
[Z,A_{J}]_{\cal T(X)} \to [Z,\holim_{S\neq \emptyset} A_S]_{\cal T(X)}
\to 
 \{Z^+,\cal W_j \smsh_X (\underset{i\in J\quad }{\smsh_X} S_X  A_i)\}_{\cal R(X)}\, .
 \]
 \end{cor}
 
 \begin{rem} Corollary \ref{useful} is a  robust generalization of
  a result of Barratt and Whitehead \cite{Barratt-Whitehead}
  and, independentally, Toda \cite{Toda}.
 \end{rem}
 
 \subsection{The Euler class} 
 Let $f\: Z \to \holim_{S\ne J} A_S$ be a map of spaces. Then
 $f$ is also a morphism of $\cal T(X)$.
 Using \ref{thm:also-hard}, we see that the composed map
 \[
 Z^+ \overset f\to \holim_{S\ne J} A_S \to C(A_\bullet) 
 \]
 gives rise to a fiberwise stable homotopy class
 \[
 e(f) \in \{Z^+,\cal W_j \smsh_X (\underset{i\in J\quad }{\smsh_X} S_X  A_i)\}_{\cal R(X)}
 \]
 which we call the {\it Euler class} of $f$. Equivalently, $e(f)$ resides in the
 cohomology group
 \[
 H^0(Z;\cal W_j \smsh_X (\underset{i\in J\quad }{\smsh_X} S_X  A_i))\, .
 \]
 
Then from Corollary  \ref{useful} we deduce

\begin{cor} \label{cor:euler} The Euler class $e(f)$ vanishes when $f$ admits a homotopy factorization through
$A_{J}$.
Conversely, when $\dim Z \le 1+\mu +\Sigma$ and $e(f) = 0$, then $f$ admits
a homotopy factorization through $A_{J}$.
\end{cor}

 \subsection{A special case\label{subsection:case-X-pt}} 
When $X=A_0$ is a point the above results can be expanded upon as follows: there is a  homotopy cofiber
sequence of spaces
\begin{equation} \label{eqn:cofibration}
A_J \to \holim \limits_{S\ne \emptyset} A_S \to C(A_\bullet)
\end{equation}
and a $(2+\mu+\Sigma)$-connected map 

\begin{equation} \label{eqn:equivalence-range}
C(A_\bullet) @>>> \Omega^\infty (W_j \smsh (\smsh_{i\in J} SA_i))\, .
\end{equation}
Furthermore, the space $A_J$ is $\mu$-connected. If we choose a basepoint
in $A_J$ then  $A_\bullet$ becomes a cube of based spaces. Let $F(A_\bullet)$ be its total homotopy fiber. By the  Blakers-Massey theorem 
applied to \eqref{eqn:cofibration} and using the map \eqref{eqn:equivalence-range}, we infer

\begin{cor} \label{cor:total-fiber}There is a 
$(1+\mu+\Sigma)$-connected map 
\[
F(A_\bullet) \to \Omega^\infty (\Sigma^{j-1} W_j \smsh (\smsh_{i\in J} A_i))
\simeq \prod_{i}^{(j-1)!} Q(\Sigma^{1-j}A_1 \smsh\cdots \smsh  A_j)\, .
\]
\end{cor}

\begin{rem} \label{rem:functoriality} The proof we give of Theorem \ref{thm:also-hard} implies that
the map of Corollary \ref{cor:total-fiber}
is natural with respect to morphisms of based cubes $A_\bullet \to B_\bullet$.
\end{rem}

\section{Proof of Theorems \ref{main-thm} and \ref{main:solution-space}\label{sec:geometry}}

In this section we give the proof of Theorems \ref{main-thm} and \ref{main:solution-space} modulo the proof of Theorem \ref{thm:also-hard}. The 
proof of the latter result
will appear in \S\ref{sec:proofs}.

Returning to the situation of \S\ref{intro}, we are given
pairwise disjoint connected closed submanifolds $Q_1,\dots,Q_j\subset N$. Let
$N\setminus Q_\bullet$ denote the $j$-cubical diagram of $\cal R(N)$ defined by
\[
S\mapsto N\setminus Q_S,\qquad S\subset J\, .
\]
Note that $N\setminus Q_\bullet$ satisfies Hypothesis
\ref{hypothesis-throughout} since $n -q_i \ge 2$.

\begin{proof}[Proof of Theorem \ref{main-thm}]
Recall that we are given a map 
\[
f\: P \to \holim_{S\subsetneq J} (N\setminus Q_S)
\] and we
wish to identify the obstructions to deforming it into $N\setminus Q_{J}$.
By transversality, the map $N\setminus Q_J \to N \setminus  Q_{J -\{i\}}$
is $(n-q_i-1)$-connected for $1\le i\le j$.  By Corollary \ref{cor:euler}, we infer

\begin{prop} \label{prop:factorize}
If $P \to \holim_{S \subsetneq J} N\setminus Q_S$ admits a homotopy
factorization through $N\setminus Q_{J}$, then $e(f) = 0$. The converse
is true provided $p \le 1+\mu+\Sigma$, where $\Sigma = \sum_i (n-q_i-2)$
and $\mu_i = \min_i (n-q_i-2)$.
\end{prop}

\begin{proof} This follows from Corollary \ref{cor:euler} since
a closed manifold $P$ of dimension $p$ 
admits the structure of a cell complex of dimension $p$.
\end{proof}

Let $\nu_i$ be the normal bundle of $Q_i$ in $N$. 
The tubular neighborhood theorem gives a
weak equivalence of $\cal R(N)$  
\[
S_N (N \setminus Q_1) \simeq D(\nu_i) \cup_{S(\nu_i)} N =: T_N(\nu_{i}) \, ,
\]
where the right side is the fiberwise Thom space of $\nu_i$ over $N$.

Stably, we can identify $\nu_i$ with the virtual bundle $\xi_i := f^*\tau_N - \tau_{Q_i}$,
given by the difference of tangent bundles. We
write $T_N(\xi_i)$ for the associated fiberwise Thom spectrum. With these notational
changes, $e(f)$ can be regarded as residing
in the cohomology group
\begin{equation} \label{eqn:cohomology-zero}
H^0(P;{\cal W}_j \smsh_N 
(\underset{i\in J\quad} {\smsh_N} T_N (\xi_i)))\, .
\end{equation}
The remainder of the proof of Theorem \ref{main-thm} 
will involve application of Poincar\'e duality \ref{thm:duality}
to the cohomology group \eqref{eqn:cohomology-zero}.

\subsection{The Euler characteristic}
By Poincar\'e duality \ref{thm:duality}, $e(f)$ corresponds to a homology class 
\[
\chi(f) \in H_0(P; {}^{-\tau_P}\!f^*({\cal W}_j   \smsh_N (\underset{i\in J\quad} {\smsh_N} T_N (\xi_i))))\, .
\]
Using the induction isomorphism (\S\ref{subsec:induction}), the group where $\chi(f)$ resides can alternatively
be written as
\[
H_0(N; f_* {}^{-\tau_P}\!f^*({\cal W}_j  \smsh_N (\underset{i\in J\quad} {\smsh_N} T_N (\xi_i))))\, .
\]
By definition, the latter is the stable homotopy group in degree
zero of the spectrum
\[
({\cal W}_j  \smsh_N T_N(-\tau_P) \smsh_N (\underset{i\in J\quad} {\smsh_N} T_N (\xi_i))) /N \, .
\]
Using Lemma \ref{Whitney-sum-formula} in virtual form, 
we deduce that the fiberwise spectrum  
\[
\cal W_j \smsh_N T_N(-\tau_P) \smsh_N (\underset{i\in J\quad} {\smsh_N} T_N (\xi_i)))
\]
can be rewritten up to homotopy as
\[
\cal W_j\smsh_N T_N(\xi)\, ,
\]
where $\xi$ is the virtual bundle over $E(P,Q_\bullet)$ that
was defined  in \S\ref{intro}.

Recall that $\cal W_j$ is just the fiberwise wedge of $(j-1)!$-copies of the fiberwise
spectrum $N \times S^{2-2j}$. From this we infer
\[
(\cal W_j \smsh_N T_N(\xi))/N \,\, \simeq\,\,  \bigvee_{(j-1)!} \Sigma^{2-2j} E(P,Q_\bullet)^\xi\, .
\]
Since
\[
\pi_0(\Sigma^{2-2j} E(P,Q_\bullet)^\xi) \cong \Omega_{2j-2}(E(P,Q_\bullet);\xi)\, ,
\]
we have deduced that the obstruction $\chi(f)$ resides in the abelian group
\[
\bigoplus_{(j-1)!}  \Omega_{2j-2}(E(P,Q_\bullet);\xi)\, .
\]
By Proposition \ref{prop:factorize}, $\chi(f)$ vanishes whenever $f\: P \to \holim_{S\subsetneq J} (N\setminus Q_S)$ admits
a homotopy factorization through $N \setminus Q_{J}$.
Conversely, if $p \le 1+\mu +\Sigma$, then $\chi(f) = 0$ we have shown
there is such a factorization of $f$.
\end{proof}

\begin{proof}[Proof of Theorem \ref{main:solution-space}]
Given a multi-relative intersection problem $f$, 
recall that the solution space $\cal L(f)$ is the space of homotopy factorizations of $f$
of the form 
\[
P \to N\setminus Q_J \to \holim_{S\subsetneq J} N\setminus Q_S\, ,
\]
where we have suppressed the lifting homotopy.
Consider the $\infty$-cocartesian square of spaces
\begin{equation} \label{eqn:solution-space-sequence}
\xymatrix{
N\setminus Q_J  \ar[r] \ar[d] & \holim_{S\subsetneq J} N\setminus Q_S \ar[d]
\\
N \ar[r] & C(N\setminus Q_\bullet)
}
\end{equation}
whose horizontal maps are $(1+\Sigma)$-connected (by the higher 
Blakers-Massey theorem applied to the $j$-cube $N \setminus Q_\bullet$ \cite[thm.~2.5]{Good}),
and whose vertical maps are $(1 + \mu)$-connected. By the  Blakers-Massey theorem,
the square is $(1+\mu+\Sigma)$-cartesian. Hence, if 
$\cal F$ is defined as the homotopy pullback of 
the diagram given by deleting $N \setminus Q_J$ from the square, then the map
$N \setminus Q_J\to \cal F$ is $(1+\mu+\Sigma)$-connected. 

Suppose that the given 
 multi-relative intersection problem comes equipped with a
 preferred solution $\hat f\: P \to N\setminus Q_J$
(where again the lifting homotopy is suppressed).
 The solution gives a preferred null-homotopy of the composite
\begin{equation} \label{eqn:null-map}
P @>f>> \holim \limits_{S\subsetneq T} N\setminus Q_S \to C(N\setminus Q_\bullet)
\end{equation}
as a morphism of $\cal T(N)$.

In other words, we have a map
\begin{equation} \label{eqn:the-map}
\cal L(f) \to \cal N(f)
\end{equation}
where $\cal L(f)$ is the solution space and $\cal N(f)$ is the space of null-homotopies 
of the composite \eqref{eqn:null-map}.
With respect to the preferred basepoint of 
$\cal L(f)$, this is a map of based spaces. 

Furthermore, 
$\cal N(f)$ can be interpreted as
the moduli space of homotopy factorizations of $f$ of the form 
\[
P \to \cal F \to \holim \limits_{S\subsetneq T} N\setminus Q_S\, .
\]
Since the map $N \setminus Q_J \to \cal F$ is $(1+\mu+\Sigma)$-connected, we infer by elementary
obstruction theory that the map $\cal L(f) \to \cal N(f)$ is $(1-p+\mu +\Sigma)$-connected.
The rest of the proof involves identifying $\cal N(f)$.

On the one hand, rather than considering null-homotopies in $\cal T(N)$, we
can equivalently add a disjoint copy of $N$ to $P$ to get
a null-homotopy in $\cal R(N)$ of the associated morphism
\begin{equation} \label{eqn:null-map+}
P^+ \to C(N\setminus Q_\bullet)\, .
\end{equation}
Then $\cal N(f)$ can be equivalently defined as the mapping space of 
null-homotopies of \eqref{eqn:null-map+} in $\cal R(N)$.

On the other hand, the (derived) mapping space
\[
\hom_{\cal R(N)}(\Sigma_N P^+,C(N\setminus Q_\bullet))
\]
acts on the space $\cal N(f)$ (this is the fiberwise analogue of the classical fact that
for a null-homotopic map of spaces $X\to Y$, the moduli space of 
null-homotopies, i.e., the space of extensions of the map 
to the cone on $X$, is a torsor
over the space of maps $\Sigma X \to Y$). The orbit of the basepoint
of $\cal N(f)$ with respect to this action gives a preferred weak equivalence
\[
\hom_{\cal R(N)}(\Sigma_N P^+,C(N\setminus Q_\bullet)) \simeq \cal N(f)\, .
\]
Using the adjunction between $\Sigma_N$ and $\Omega_N$, we infer that there is
a preferred $(1-p+\mu+\Sigma)$-connected (weak) map
\begin{equation}\label{eqn:almost-there}
\cal L(f) \to \hom_{\cal R(N)}(P^+,\Omega_N C(N\setminus Q_\bullet)) \, .
\end{equation}
By Theorem \ref{thm:also-hard}  
we also have a $(2+\mu+\Sigma)$-connected map
\[
C(N\setminus Q_\bullet) \to  \Omega^\infty_N (\cal W_j \smsh_N T_N(\xi))\, .
\]
Applying to the latter the fiberwise loop functor
$\Omega_N$, then applying $\hom_{\cal R(N)}(P^+,{-})$, 
and composing with \eqref{eqn:almost-there}
we get a  $(1-p+ \mu+\Sigma)$-connected (weak) map
\begin{equation} \label{eqn:change-target}
\cal L(f) \to \hom_{\cal R(N)}(P^+,\Omega^{\infty+1}_N (\cal W_j \smsh_N T_N(\xi))) \, .
\end{equation}
By definition, the target of the map \eqref{eqn:change-target} is identified with the 
infinite loop space associated with the cohomology spectrum
\[
H^\bullet(P;\Sigma^{-1}_N\cal W_j \smsh_N T_N(\xi))\, .
\]
By the Poincar\'e duality argument appearing in the proof of Theorem \ref{main-thm} above, this spectrum is weakly equivalent to 
\[
\bigvee_{(j-1)!} E(P,Q_\bullet)^{\xi+(1-2j)\epsilon}\, .
\]
Assembling, we have produced a $(1-p+\mu+\Sigma)$-connected (weak) map
\[
\cal L(f) \to \prod_{(j-1)!}\Omega^{\infty} E(P,Q_\bullet)^{\xi+(1-2j)\epsilon})\, .
\]
This completes the proof of Theorem \ref{main:solution-space}.
\end{proof}

\subsection{The euclidean case}
When $N=\Bbb R^n$ we have a  corollary to Corollary
\ref{cor:total-fiber}. Consider an embedding $Q_J \subset \Bbb R^n$
where now each $Q_i$ is a manifold admitting a handle decomposition with 
handles having index at most $q_i$ where $n-q_i \ge 3$.

Consider the $j$-cubical diagram
$\Bbb R^n \setminus Q_\bullet$. Choose a basepoint in $\Bbb R^n \setminus
Q_J$. Then the $j$-cube is based and we consider its total homotopy fiber,
\[
\Phi(\Bbb R^n \setminus Q_\bullet)\, .
\] 
For $A\subset \Bbb R^n$ let $A^\ast = \Bbb R^n \setminus A$ denote
its complement.

\begin{cor} \label{cor:important} There is a $(1+\mu+\Sigma)$-connected map
\begin{equation} \label{eqn:important}
\Phi(\Bbb R^n \setminus Q_\bullet) @>>> \prod_{i=1}^{(j-1)!} 
\Omega^\infty(\Sigma^{1-j} Q_1^\ast \smsh \cdots \smsh  
Q_j^\ast)
\end{equation}
where $\mu = \min_i(n-q_i-2)$ and $\Sigma = \sum_i (n-q_i-2)$\end{cor}


\begin{rem}\label{rem:important} The target of the map \eqref{eqn:important} may also be identified with the infinite loop space
associated with the wedge of  $(j-1)!$-copies of the spectrum
\[
\Sigma^{1-jn} D_+(Q_1\times \cdots \times Q_j)
\]
where  $D_+(X) = F(X_+,S^0)$ is the Spanier-Whitehead dual of $X_+$. 
\end{rem}

\section{Proof of Theorem \ref{thm:also-hard}\label{sec:proofs} }

The proof of  Theorem \ref{thm:also-hard} relies on basic
results arising in the calculus of the identity functor
 which we now summarize.
Let 
\[
\mathbb I\:  \cal T \to \cal T
\]
be the identity functor. By \cite{Good3}
 one has a tower of natural transformations
\[
\cdots \to P_2\mathbb I \to P_1\mathbb I   \to P_0 \mathbb I  = \ast
\]
and compatible natural transformations $\Bbb I \to P_j\mathbb I $.
Furthermore the functor $P_j\Bbb I$ is $j$-excisive in the sense that it transforms strongly cocartesian $(j+1)$-cubes into $\infty$-cartesian ones.
In what follows, we abbreviate notation by setting
$P_j := P_j\mathbb I $.

If $Y$ is $r$-connected, then the map $Y \to P_jY$ 
is $(jr+1)$-connected. In particular, when $r > 0$, the map
\[
Y \to \lim_{j\to\infty} P_jY
\]
is a weak homotopy equivalence. 

If $Y$ is a based space, then
the $j$-th layer of the tower, that is the homotopy fiber of
 $P_jY \to P_{j-1}Y$, 
is isomorphic in the homotopy category of functors to the infinite loop space valued functor
\[
Y \mapsto \Omega^\infty \Bbb D_j Y\, ,
\]
where $\Bbb D_j$ takes values in spectra. 

The functor $\Bbb D_j$
is classified by
a certain spectrum with $\Sigma_j$-action, 
denoted $L_j$, whose underlying homotopy type is that of  a wedge of $(j-1)!$ copies of the $(1-j)$-sphere spectrum \cite{Johnson}, \cite[p.~706]{Good3}. Then
\begin{equation} \label{eqn:id-coeff}
\Bbb D_j Y \, \, \simeq \,\, L_j \smsh_{h\Sigma_j} Y^{[j]}\, ,
\end{equation}
where $ Y^{[j]}$ denotes the $j$-fold fiberwise smash product 
$Y$.  This description of $\Bbb D_j $ enables one to
extend its domain of definition to the category of spectra, i.e., if $A$ is a spectrum then $\Bbb D_jA$ is the spectrum
$L_j \smsh_{h\Sigma_j} A^{[j]}$.

\begin{rem} The maps of the tower
$P_j Y \to P_{j-1}Y$ are principal fibrations in the sense that
there is a homotopy fiber sequence 
\[
P_jY \to P_{j-1} Y \to B D_jY\, ,
\]
where $B D_j Y$ is the delooping of $D_jY$ given by 
$\Omega^\infty(\Sigma \Bbb D_j Y)$ (cf. \cite[p.~653]{Good3}).
\end{rem}

We now  consider the strongly cocartesian $j$-cube 
$A_\bullet$ of $\cal R(X)$. Assume for now that $X$ is contractible.
Without loss in generality we can replace $X$ by the one point space. The assignment $S\mapsto P_kA_S$ defines a $j$-cube denoted $P_kA_\bullet$. A choice of basepoint in $A_J$ equips $A_\bullet$ with the structure of a based $j$-cube. Then $D_k A_\bullet$ is
a $j$-cube of infinite loop spaces. Let 
\[
\text{fib}(D_k(A_\bullet))
\]
denote its total homotopy fiber.

\begin{prop} \label{prop:key-decomp} The total homotopy fiber of $D_k A_\bullet$ is $(\mu+\Sigma)$-connected if $k \ge j+1$. Furthermore, 
when $k=j$ there is a $(1+\mu+\Sigma)$-connected map
\[
\text{\rm fib} (D_j(A_\bullet)) \to \Omega^\infty (L_j \smsh A_1 \smsh \cdots \smsh A_j) \, .
\]
\end{prop}

\begin{proof} 
Suppose first that $A_\bullet$ is a wedge cube on 
the based spaces $X_1,\dots,X_j$. Then $X_i$ is
$r_i$-connected. 
Using \eqref{eqn:id-coeff}, the total homotopy fiber of $D_k(A_\bullet)$ may be identified with the infinite loop space associated with
 total homotopy fiber of the $j$-cube of spectra 
\begin{equation} \label{eqn:jcubexs}
 S\mapsto L_k \smsh_{h\Sigma_k} X_S^{[k]}\, ,
\end{equation} 
 where $X_S$ is the wedge of the spaces $X_i$ for $i\in S$.
Applying the binomial theorem to expand $X_S^{[k]}$, direct calculation shows that the total homotopy fiber of \eqref{eqn:jcubexs}
decomposes into a wedge  of terms of the form
 \begin{equation}\label{eqn:decomp-wedge-smash}
L_k \smsh_{h\Sigma_{s_\bullet}} 
(X_1^{[s_1]} \smsh \cdots \smsh X_j^{[s_j]}) \, ,
 \end{equation}
 where 
\begin{itemize}
\item $\sum_i s_1 = k$ with $s_i \ge 1$ for all $i$.
\item $\Sigma_{s_\bullet} := 
\Sigma_{s_1} \times \cdots \times \Sigma_{s_j} \subset \Sigma_k$. 
\end{itemize}
If $k \ge j+1$ then there is always at
least one term $s_i \ge 2$.
It follows that the displayed spectrum is at least 
$(\mu+\Sigma)$-connected. Hence,  the total homotopy fiber $\text{fib} (D_k(A_\bullet))$ is
also $(\mu+\Sigma)$-connected when $k \ge j+1$.  

When $k=j$, 
we can ignore those terms in which $s_i \ge 2$ since they are highly connected: 
the projection away from those terms produces the
$(1+\mu+\Sigma)$-connected map
\[
\text{fib} (D_j(A_\bullet)) \to \Omega^\infty (L_j \smsh A_1 \smsh \cdots \smsh A_j)\, .
\]
This completes the proof in the case of wedge cubes.

Turning to the general case, we use the fact that $\Bbb D_k$ 
is defined on the category of spectra. 
By Lemma \ref{lem:wedge-cube}, the
$j$-cube of spectra $\Sigma^\infty A_\bullet$ is a weakly equivalent to a wedge cube on the spectra $\Sigma^\infty A_1,\dots,
\Sigma^\infty A_j$. Replacing the spaces $X_i$ of the previous case by the spectra $\Sigma^\infty A_i$ and making 
the same kind of calculation, the conclusion follows.
\end{proof}

\begin{cor} \label{cor:k+1tok} Assume that $X$ is
contractible and $k \ge j+1$. Then  $(j+1)$-cube
\[
P_k A_\bullet \to P_{k-1} A_\bullet
\]
is $(1+\mu + \Sigma)$-cartesian.
\end{cor}

\begin{prop} \label{prop:AtoP_j} Assume that $X$ is
contractible. 
Then the $(j+1)$-cube
\[
A_\bullet \to P_j A_\bullet
\]
is $(1+ \mu+\Sigma)$-cartesian.
\end{prop}

\begin{proof}
  If $r_i \ge 1$ for all $i$, the result follows easily from 
induction, Corollary \ref{cor:k+1tok} and the convergence
of the  tower for the identity functor for 1-connected spaces. 
In the general case one must proceed differently using
the higher Blakers-Massey theorem.
We  are indebted to the referee for communicating
 the following argument.

We first recall how  $Y\mapsto P_jY$ is defined in terms of an auxiliary
functor $Y\mapsto T_jY$ as in \cite[\S1]{Good3}.
The latter  is given by the taking
 homotopy limit of the functor
\[
U \mapsto Y\ast U
\]
where $\ast$ means topological join and $U$ ranges over the poset of nonempty subsets of $\{1,\dots,j+1\}$. There is an evident natural transformation $Y \to T_jY$ and $P_jY$ is defined to be the homotopy colimit of the diagram 
\[
Y \to T_jY \to T_j^2 Y \to \cdots\, .
\]
For the rest of the proof we set 
$\underline k = \{1,2,\dots,k\}$ to avoid notational clutter.

We first determine how cartesian
the $(j+1)$-cube $A_\bullet \to T_jA_\bullet$ is. 
This is the same as asking the degree to which the $(2j+1)$-cube
\[
(T,U) \mapsto A_T \ast U
\]
is cartesian where $T\subset \underline j$ and $U \subset \underline{j+1}$ (note: by our conventions this functor is contravariant in the first variable and covariant in the second).

 For fixed $T$,  the 
$(j+1)$-cube $U \mapsto A_T \ast U$ is strongly cocartesian.
Similarly, for fixed $U$, the $j$-cube $T \mapsto A_T \ast U$ is
strongly cocartesian. Any pair $(T,U)$ corresponds to subcube whose initial term is $A_T\ast U$. It follows that
this subcube will be $\infty$-cocartesian whenever 
$|T| \ge 2$ or $|U|\ge 2$.  Consequently, there are three remaining types of pairs $(T,U)$ to consider:
\begin{enumerate}
\item $|T| = 1, |U| = 0$;  
\item $|T| = 0, |U|=1$;
\item $|T|=|U| = 1$;
\end{enumerate}
By inspection, one finds for a type (1) pair that the subcube is 
$(r_i+1)$-cocartesian. Similarly, for a type (2) pair the  subcube is $(\mu+1)$-cocartesian and for a type (3) pair the subcube is
$(r_i+2)$-cocartesian.  

Given a partition of $\underline j \amalg \underline{j+1}$ consisting of sets of these types only, the sum of these numbers indexed over the sets of the partition is given by 
\begin{equation} \label{eqn:partition-sum}
\Sigma +j + D + (j+1-D)(\mu+1)\, ,
\end{equation}
where $D$ is the number of times a set of type (3)
occurs in the partition. To see this, note that
any such partition is determined by  a choice of injections
$a\: \underline D \to \underline j, b\:\underline D \to \underline j$,
in which the complement of image of $a$ defines the
type (1) singletons of the partition and the complement of the image of $b$ defines the singletons of type (2).
Hence, the sum of the numbers for such a partition is given by
\[
\sum_{i \notin a(\underline D)} (r_i + 1) + 
\sum_{i \notin b(\underline D)} (\mu_i + 1)  +
\sum_{i \in a(\underline D)} (r_i+2)\, ,
\]
which clearly coincides with the expression \eqref{eqn:partition-sum}.

Observe that \eqref{eqn:partition-sum} achieves a minimum when $D$ is at its maximal value $j$.  It follows that the minimal value is $1+\mu + \Sigma + 2j$. 
Since we are dealing with a $(2j+1)$-cube, we subtract $2j$ to get
$1+\mu + \Sigma$, which is how cartesian the cube is by
\cite[th.~2.5]{Good}. Hence, the $(j+1)$-cube
$A_\bullet \to T_jA_\bullet$ is $(1+\mu + \Sigma)$-cartesian.

The next step is to consider 
$T_j A_\bullet \to T_{j+1}^2A_\bullet$.
For each fixed non-empty  $U \subset \underline{j+1}$, the 
map of $j$-cubes
\[
A_\bullet \ast U \to T_j(A_\bullet \ast U)
\]
is of the kind we considered above with the number $r_i$ increased by one (so $\Sigma$ is increased by $j$) and $\mu$ increased by $1$.
 Hence, the corresponding $(j+1)$-cube
is $(1+ (1+\mu)+(\Sigma+j))$-cartesian. Moreover, 
taking the homotopy limit over $U$
yields the map of $j$-cubes
$T_j A_\bullet \to T_j^2A_\bullet$. 
In taking this homotopy limit the degree to which
the latter is cartesian 
is decreased by $j$. We infer that $T_j A_\bullet \to T_j^2A_\bullet$ is $(2+\mu+\Sigma)$-cartesian, which is one better
than the estimate we obtained for $A_\bullet \to T_jA_\bullet$.
Repeating this argument, we infer that 
$
T^k_j A_\bullet \to T_j^{k+1}A_\bullet
$
is $(k+1+\mu+\Sigma)$-cartesian for any $k \ge 0$. It follows
that $A_\bullet \to P_jA_\bullet$ is $(1+\mu+\Sigma)$-cartesian.
\end{proof}



\begin{proof}[Proof of Theorem  \ref{thm:also-hard}]
The proof is a verification in two cases.
\medskip

\noindent {\it Case 1:} $X$ is contractible.
There is no loss in generality in assuming that $X$ is a point. Equip $A_J$ with a basepoint. Then
$A_\bullet$ is a $j$-cube of $1$-connected based spaces.

Consider the  commutative diagram
\[
\xymatrix{
P_jA_J \ar[d]_{b_1} \ar[r]^{a_1} & 
P_{j-1} A_J \ar[d]_{b_2}^{\simeq}
\ar[r]^{a_2} & BD_jA_J \ar[d]^{b_3} \\
\holim_{S \subsetneq  J}   P_j A_S\ar[r]_{a_3} 
& 
 \holim_{S \subsetneq  J}   P_{j-1} A_S 
 \ar[r]_{a_4} &  \holim_{S \subsetneq  J} 
 BD_jA_S 
}
\]
in which the top and bottom rows form fibration sequences.
The map $b_2$ is a homotopy equivalence since $P_{j-1}$ is $(j-1)$-excisive.
The map $b_3$ is equivalent to a principal fibration in the following sense: it may be identified with  the map of infinite loop spaces arising from the map of spectra
\[
\Sigma \Bbb D_j(A_\bullet) \to \holim_{S \subsetneq  J} 
\Sigma \Bbb D_j(A_S)
\]
associated with the
$j$-cube $\Sigma \Bbb D_j(A_\bullet)$. 

Set $W_j := \Sigma^{1-j} L_j$.
By Proposition \ref{prop:key-decomp}
there is 
a $(2+\mu+\Sigma)$-connected map of spectra 
\begin{equation} \label{eqn:really-good-map}
\Sigma \text{fib}(\Bbb D_j(\Sigma^\infty A_\bullet)) \to  
W_j \smsh SA_1 \smsh \cdots \smsh SA_j\, ,
\end{equation}
where we have implicitly made the identification
$\Sigma (L_j \smsh A_1 \smsh \cdots \smsh A_j)   \simeq
W_j \smsh SA_1 \smsh \cdots \smsh SA_j$ to avoid
displaying the choice of basepoint. 
The infinite loop space associated with the source of \eqref{eqn:really-good-map} is
identified with 
the homotopy fiber of the map $b_3$.

Consequently, 
\[
BD_jA_J @> b_3 >>  \holim_{S \subsetneq  J} 
 BD_jA_S  @>>> 
 \Omega^\infty (\Sigma W_j \smsh SA_1 \smsh \cdots \smsh SA_j) \, .
 \]
is a homotopy fiber sequence in degrees $\le 2+\mu+\Sigma$.

Hence, by Lemma \ref{lem:tech} below there is a homotopy fiber sequence in degrees $\le 1+\mu+\Sigma$ of the form
\begin{equation}\label{eqn:pj-sequence}
P_jA_J @> b_1 >> \holim_{S \subsetneq  J}   P_j A_S \to  \Omega^\infty (W_j \smsh SA_1 \smsh \cdots \smsh SA_j) \, .
\end{equation}
According to Proposition \ref{prop:AtoP_j}, the square
\[
\xymatrix{
A_J \ar[r]\ar[d] & P_jA_J \ar[d]^{b_1} \\
 \holim_{S \subsetneq  J}  A_S \ar[r] &  \holim_{S \subsetneq  J}   P_{j} A_S 
 }
 \]
 is $(1+\mu+\Sigma)$-cartesian. Let 
 $\holim_{S \subsetneq  J}  A_S \to \Omega^\infty (W_j \smsh SA_1 \smsh \cdots \smsh SA_j)$ be the composition of the bottom map of the square with the second map of \eqref{eqn:pj-sequence}.
Then 
 \begin{equation} \label{eqn:a-sequence}
 A_J \to \holim_{S \subsetneq  J}  A_S \to \Omega^\infty (W_j \smsh SA_1 \smsh \cdots \smsh SA_j)
 \end{equation}
 is also a homotopy fiber sequence in degrees $\le 1+\mu+\Sigma$.
 By the dual Blakers-Massey theorem, we conclude that \eqref{eqn:a-sequence} is
 also a homotopy cofiber sequence in degrees $\le 2+\mu+\Sigma$.
 
Consequently, the induced map
 \[
 C(A_\bullet) \to \Omega^\infty (W_j \smsh SA_1 \smsh \cdots \smsh SA_j)
 \]
 is $(2+\mu+\Sigma)$-connected.
 \medskip

 \noindent {\it Case 2:} $X$ is general.  Let $\tilde X \to X$ be a
 universal principal bundle for $X$ with topological structure group $G$. Then $\tilde X$ is contractible. Let $\tilde A_\bullet$ be the strongly
 cocartesian $j$-cube of $G$-spaces given by the fiber product
 \[
 \tilde A_S := \tilde X \times_X  A_S\, .
 \]
 The terminal vertex of this cube is then contractible, and one checks that the argument in Case 1 preserves equivariance.  It follows that there is a $(2+\mu+\Sigma)$-connected
map of based $G$-spaces
\begin{equation}\label{eqn:equivariant}
C(\tilde A_\bullet) \to 
\Omega^\infty (W_j \smsh S\tilde A_1 \smsh \cdots \smsh S\tilde A_j)\, .
\end{equation}
The result follows  by applying the Borel construction
${-}\times_G \tilde X$ to \eqref{eqn:equivariant} to obtain
a $(2+\mu+\Sigma)$-connected
map of $\cal R(X)$
\[
C(A_\bullet) \to  \Omega^\infty_X (\cal W_j \smsh_X S_X A_1 \smsh \cdots \smsh_X S_X A_j)\, . \qedhere
\]
\end{proof}

The section ends with an elementary result about fibrations that was used in the proof of Theorem  \ref{thm:also-hard}.
Let
\[
\xymatrix{
F_1 \ar[r]\ar[d] & E_1 \ar[r]\ar[d]^{\simeq} & B_1  \ar[d] \\
F_2 \ar[r] & E_2 \ar[r] & B_2 \\
}
\]
be a commutative diagram of connected spaces 
in which the rows are
fibration sequences
 and where the map $E_1 \to E_2$ is a homotopy equivalence. 
 Here $B_1\to B_2$ is a map of based spaces and the fiber
 over the basepoint of $B_i$ is $F_i$.

\begin{lem}\label{lem:tech} Assume in addition that
the map $B_1\to B_2$ sits in a homotopy
fiber sequence $B_1 \to B_2 \to B_3$ in degrees $\le s$. 
Then the map $F_1 \to F_2$ sits in a homotopy fiber sequence
$
F_1 \to F_2 \to \Omega B_3
$
in degrees $\le s-1$.
\end{lem}

\begin{proof} Equip $B_3$ with the basepoint from $B_2$. The composition $E_1 \to E_2 \to B_2 \to B_3$ is null homotopic. Hence 
$E_2 \to B_2 \to B_3$ is also null homotopic.  
Let $E_2\to PB_3$ be adjoint to a null homotopy, where $PB_3$ is
the based path space. Then the diagram 
\[
\xymatrix{
E_1 \ar[r]\ar[d] & E_2 \ar[r]\ar[d] & PB_3 \ar[d]\\
B_1 \ar[r] & B_2  \ar[r] & B_3
}
 \]
commutes. The result follows by taking fibers vertically.
\end{proof}

\section{Multiple disjunction \label{sec:disj}}

Let $P$, $Q_1,\dots, Q_j$ and $N$ be as in \S\ref{intro}. Let
\[
E(P,N)
\] denote the space of smooth embeddings from $P$ to $N$.
Then $S\mapsto E(P,N\setminus Q_S)$ forms a $j$-cube of spaces, denoted $E(P,N\setminus Q_\bullet)$.
The natural transformation from embeddings to functions
\begin{equation} \label{eqn:ef}
E(P,N\setminus Q_\bullet) \to F(P,N\setminus Q_\bullet)
\end{equation}
is a map of $j$-cubes.  One of the main results of \cite{GK2} is:

\begin{thm}[{\cite[thm.~E]{GK2}}] \label{thm:mult-disjunct} Assume $p,q_i \le n-3$.  Then the $(j+1)$-cube 
\eqref{eqn:ef} is $(n-2p-1+\Sigma)$-cartesian.
\end{thm}

\begin{proof}[Proof of Theorem \ref{embedding-result}]
Let 
\[
f\in \holim_{S\neq J} \emb(P,N\setminus Q_S)
\]
be any point. Then $f$ is represented by
a map of $(j+1)$-ads 
\[
\Delta^{j-1} \to E(P,N)\, ,
\] 
where  the $i$-th face of $\Delta^{j-1}$ is constrained 
to map into the subspace $E(P,N\setminus Q_{i+1})$
for $i = 0,1,\dots,j-1$. Note that by forgetting information, we may also regard $f$
as a map $P \to \holim_{S\neq J} N\setminus Q_S$, and therefore we have an associated
multi-relative intersection problem. 
Consequently, Theorem \ref{embedding-result} follows
by combining Theorem \ref{thm:mult-disjunct} with Theorem \ref{main-thm}. \end{proof}
 
\begin{rem} Theorem \ref{embedding-result} is a multi-relative version of 
\cite[thm.~2.2]{H-Q}.
\end{rem}

\section{The embedding tower \label{sec:speculation}}

In \cite[\S13]{KW1} and \cite{KW2}, we described an invariant $\mu(f)$ 
which was shown to be a complete obstruction to regularly homotoping an immersion $f\: P \to N$ to an
 embedding in the metastable range. 
 The goal of this section is to generalize this result beyond the metastable range when $N = \Bbb R^n$ is euclidean space.  

\subsection{Construction of the embedding tower} 
Let  $P$ be a smooth manifold of dimension $p$ without boundary and let $N$ be a smooth manifold of dimension $n$. We let $E(P,N)$ denote the space
of embeddings of $P$ in $N$, defined as the geometric realization of
the simplicial set whose $k$-simplices are the smooth families of
embeddings from $P$ to  $N$ that are parametrized by the standard $k$-simplex.

Assume $P$ is compact.
Let $\cal O_j := \cal O_j(P)$ be the partially ordered set whose elements  are  open subsets $U \subset P$ such that
$U$ is diffeomorphic to $\Bbb R^p \times T$, where $T$ is a set of cardinality at most $j$. A morphism $U \to V$ is given by an inclusion of subsets. The  $j$-th stage of the  Goodwillie-Weiss embedding tower is defined by 
\[
E_j(P,N) := \holim_{U \in \cal O_j} E(U,N)\, .
\] 
The inclusion $\cal O_{j-1} \to \cal O_j$ induces a map
$E_j(P,N) \to E_{j-1}(P,N)$. The map $E(P,N) \to E_j(P,N)$
is given by restricting  embeddings
to elements of $\cal O_j$.

If $p \le n-1$ then $E_1(P,N)$ is homotopy equivalent to $I(P,N)$, the space of immersions from $P$ to $N$ (by a reformulation of Smale-Hirsch theory). Hence,  a basepoint of $E_1(P,N)$
amounts to selecting an immersion $P \to N$ up to contractible choice.
In what follows, we fix such a basepoint and define
\[
\bar E_j(P,N) := \text{fiber}(E_j(P,N) \to E_1(P,N))\, .
\]
It follows that the square
\begin{equation} \label{eqn:bar-unbar}
\xymatrix{
\bar E_j(P,N) \ar[r] \ar[d] & E_j(P,N) \ar[d] \\
\bar E_{j-1}(P,N) \ar[r] & E_{j-1}(P,N)
 }
 \end{equation}
 is homotopy cartesian. Furthermore, the tower $\{\bar E_j(P,N)\}$
 is the manifold calculus tower associated with 
 the functor $U \mapsto \bar E(U,N)$, where $U$ varies throughout the open subsets of $P$. Call this the {\it reduced embedding tower}.
 Note that $\bar E_1(P,N)$ is the one-point space.
 
 \subsection{Configuration spaces} For a set $J$ of cardinality $j$,
set
\[
E_J(N)  := E(J,N)\, .
\] 
If we equip $J$  with a total ordering, then
$E_J(N)$ is 
the configuration space
of finite ordered subsets of $N$ of cardinality $j$. 
A choice of embedding $J \to N$ equips $E_J(N)$ 
with a baspoint.
To each $T\subset U\subset J$ there is  a projection map
$E_U(N) \to E_T(N)$. These assemble into a $j$-cube of based spaces 
$E_\bullet(N)$.

\begin{lem}\label{lem:cartesian-config}
The $j$-cube $E_\bullet(N)$ is $(j-1)(n-2)+1$-cartesian.
\end{lem}

\begin{proof} 
The $j$-cube $E_\bullet(N)$ can be written as a map of 
 $(j-1)$-cubes
 \[
 E_{S\cup 1}(N) \to E_S(N)
 \]
 where $S\subset J_1 := \{2,\cdots,j\}$. 
 The displayed map is a
 fibration  whose fiber at the basepoint is the based space $N\setminus S$.  These form a strongly cocartesian
 $(j-1)$-cube $N_\bullet$, all of whose maps 
 are $(n-1)$-connected. Then $N_\bullet$ is $(j-1)(n-2)+1$-cartesian
 by the higher Blakers-Massey theorem. 
\end{proof}

\subsection{The unstable obstruction}
For $j \ge 2$ let
$
\binom{P}{j}
$
denote the configuration space of subsets $S\subset P$ of cardinality $j$.  
Over this space we consider two fibrations.
The first fibration
\[
E \to \tbinom{P}{j}
\]
has fiber over $S\in \binom{P}{j}$
 given by  the configuration space $E_S(N)$. 
 
The second fibration
\[
D \to \tbinom{P}{j}
\]
has fiber over $S$ given by
$\holim_{T\subsetneq S} E_T(N)$.

Then one has an evident map of fibrations
\begin{equation} \label{eqn:map-fibration-config}
E \to D\, .
\end{equation}
A point $x\in E_{j-1}(P,N)$  determines a section $t = t(x)$
of $D \to \binom{P}{j}$. It also determines a 
partial section $s = s(x)$ of $E  \to \binom{P}{j}$ along an open collar  of the boundary of a compactification of $\binom{P}{j}$.  The sections agree with respect to the map \eqref{eqn:map-fibration-config}.

The following is essentially just a reformulation of Weiss' description of the layers of the embedding tower.

\begin{lem} \label{lem:lift-lemma} Assume $j \ge 2$. The
homotopy fiber of $\bar E_j(P,N) \to \bar E_{j-2}(P,N)$ taken at $x$
is homotopy equivalent to the space of sections of 
$E \to \binom{P}{j}$ which are compatible with $t$
and which coincide with $s$ near infinity.
In particular, $x$ lifts to a point of $E_j(P,N)$ if and only if this section space is non-empty.
\end{lem} 

\begin{rem}  Another formulation of the lemma is that the square
\[
\xymatrix{
\bar E_j(P,N) \ar[r]\ar[d] & \Gamma(E)\ar[d]\\
\bar E_{j-1}(P,N) \ar[r] & \Gamma_\infty(E)
\times_{\Gamma_\infty(D)} \Gamma(D)
}
\]
is a $\infty$-cartesian, where $\Gamma$ denotes the space of sections and $\Gamma_\infty$ denotes the space of germs of sections near infinity. 
\end{rem}

\begin{proof}[Proof of Lemma \ref{lem:lift-lemma}]  
Given $x$, define a third fibration 
\[
F \to \tbinom{P}{j}
\]
whose fiber at $S$  
 is the total homotopy fiber of the cube
 $T \mapsto E_T(N)$ for $T\subset S$. Denote this fiber by 
 $\Phi_S(N;x)$.  It is an unbased space.
 Note that $\Phi_S(N;x)$
 is well-defined since when $T\subsetneq S$, 
each of the spaces
 $E(T,N)$ is based using $x$. 
 
 Moreover $x$ gives a partial section of this fibration at infinity.  Weiss shows that the space of compactly supported sections of this fibration (i.e., the space of sections agreeing with the partial section near infinity) coincides with the homotopy fiber of $\bar E_j(P,N) \to \bar E_{j-1}(P,N)$  at $x$.
The latter space is homotopy equivalent to the space in the statement of the lemma.
\end{proof}

\subsection{A cohomological obstruction}
If we  suspend the fibers of $D \to \binom{P}{j}$,
then the obstruction to finding a compactly supported section lies in a spectrum cohomology group. If certain dimensional restrictions
are present, then nothing is lost in suspending.

When $X$ is an unbased space, we define its suspension spectrum
be the homotopy fiber of the map of spectra
$\Sigma^\infty X_+ \to S^0$ that is induced by the map from $X$ to the one point space. By slight abuse in notation, denote the
 homotopy fiber by $\Sigma^\infty X$.

\begin{defn} Let 
\[
\cal D \to \tbinom{P}{j}
\]
be the fiberwise spectrum whose fiber at $S$ given by 
$\Sigma^\infty \Phi_S(N;x)$. This comes equipped with a section near infinity. Note that $\cal D$ depends on the choice of
$x$. 
\end{defn}

The total obstruction $e(x)$ to finding  a compactly supported section 
of $\cal D$
lies in $\pi_{-1}$ in the spectrum of compactly supported sections, that is 
\[
e(x) \in H^{-1}_{\text{\rm cs}}(\tbinom{P}{j};\cal D)\,  .
\]

\begin{lem}\label{lem:e-vanish} If $x\in E_{j-1}(P,N)$ lifts to $E_j(P,N)$, then
$e(x)$ vanishes. The converse is true provided that 
$2(j-1)(n-2) - jp + 1 \ge 0$.
\end{lem}

\begin{proof} The if part is clear. For the converse,
one observes that the map $\Phi_S(N;x) \to \Omega^\infty \Sigma^\infty \Phi_S(N;x)$ is $(2(j-1)(n-2) + 1)$-connected 
using the Freudenthal suspension theorem 
and fact that $\Phi_S(N;x)$ is $(j-1)(n-2)$-connected by 
Lemma \ref{lem:cartesian-config}.
It follows that the map of compactly supported section spaces
is $(2(j-1)(n-2)-jp + 1)$-connected.
\end{proof}

\subsection{Highly connected manifolds}
When $N$ is highly connected, the obstruction
to lifting simplifies considerably.

\begin{defn}
For $S\subset \binom{P}{j}$ let
\[
C_S(N)
\]
denote the mapping cone of the map 
\[
E_S(N)) \to \holim_{T\subsetneq S} E_T(N)\, .
\]
\end{defn}

\begin{rem}
In contrast with $\Phi_S(N;x)$, the space
$C_S(N)$ doesn't depend on $x$ and it has a preferred basepoint.\end{rem}

\begin{lem}\label{lem:key-cart} Assume $j \ge 2$ and $N$ is $r$-connected,  where $r \le n-2$. Then the square 
\[
\xymatrix{
E_S(N) \ar[r]\ar[d] & \holim_{T\subsetneq S} E_T(N) \ar[d] \\
C \ar[r] &  C_S(N)
}
\]
is $(j-1)(n-2)+r+1)$-cartesian, where $C$ is the cone on $E_S(N)$.
\end{lem}

\begin{proof} By definition, the square is $\infty$-cocartesian.
Furthermore, the map 
$E_S(N) \to \holim_{T\subsetneq S} E_T(N)$
is $((j-1)(n-2)+1)$-connected by Lemma \ref{lem:cartesian-config}. 

Since $N$ $r$-connected and $r\le n-2$, it follows
 that  $E_S(N)$ is $r$-connected. Hence the left vertical map
 is $(r+1)$-connected.
The conclusion now 
follows from the Blakers-Massey theorem.
\end{proof}

Let 
\begin{equation}\label{eqn:q-fibration}
\cal C \to \tbinom{P}{j}
\end{equation}
 be the fiberwise spectrum whose fiber at $S$ is 
 $\Sigma^\infty C_S(N)$. 
 This fiberwise spectrum doesn't depend on $x$.
 
 The section $t$ induces another section of \eqref{eqn:q-fibration}, call it $t'$. The latter section is homotopic to the zero section near infinity. Then an obstruction to 
 lifting $x\in E_{j-1}(P,N)$ to $E_j(P,N)$ is given  by the associated compactly supported spectrum cohomology class of $t'$:
 \[
 e'(x) \in H^0_{\text{cs}}(\tbinom{P}{j}; \cal C)\, .
 \]

\begin{lem} \label{lem:euclid} Assume $j \ge 2$ and $N$ is $r$-connected with $r\le n-2$.
If $x$ lifts to an element of $E_j(P,N)$, then
$e'(x)$ vanishes. Furthermore, the converse holds if
 $r\ge p-1-(j-1)(n-p-2)$.
\end{lem}

\begin{proof}  The proof uses the commutative square
\[
\xymatrix{
\Sigma\Phi_S(N;x) \ar[r] \ar[d] & C_S(N) \ar[d] \\
\Omega^\infty \Sigma^\infty \Phi_S(N;x) \ar[r]  & \Omega^\infty \Sigma^\infty C_S(N) 
}
\]
Since $C_S(N)$ and $\Sigma\Phi_S(N;x)$ 
are $((j-1)(n-2)+1)$-connected (by Lemma \ref{lem:cartesian-config})
 the vertical maps are  $(2(j-1)(n-2)+3)$-connected by the Freudenthal suspension theorem.
 
By Lemma \ref{lem:key-cart}, the 
horizontal maps are 
$((j-1)(n-2)+r+2)$-connected. Hence,
the composite  
\[
\Sigma\Phi_S(N;x) \to C_S(N)\to \Omega^\infty\Sigma^\infty C_S(N)
\]
is $((j-1)(n-2)+r+2)$-connected.
By elementary obstruction theory the obstructions
$e'(x)$ and $e(x)$ contain the same information when $jp < (j-1)(n-2)+r+2$, that is, when $r\ge p-1-(j-1)(n-p-2)$.
\end{proof}

\begin{cor} \label{cor:euclid} Assume $j\ge 2$. If $N$ is contractible, then $x$ lifts to an element of $E_j(P,N)$ if and only if $e'(x)=0$.
\end{cor}

\begin{proof} In this case we can take $r=n-2$. Then the inequality of Lemma \ref{lem:euclid} becomes $n-2 \ge p-1-(j-1)(n-p-2)$
which is automatically satisfied because $p \le n-3$.
\end{proof}

\subsubsection{Equivariant reformulation} 
Set $J := \{1,\dots,j\}$. Then the map $E_J(P) \to \binom{P}{j}$
which assigns to an embedding its image is a regular covering space
with structure group $\Sigma_j$, where the latter acts on $E_J(P)$
via the automorphisms of $J$.

The pullback of $\cal C \to \tbinom{P}{j}$ along 
 $E_J(P) \to \tbinom{P}{j}$ coincides with the
 fiberwise spectrum with $\Sigma_j$-action
\begin{equation}\label{eqn:nottrivial}
E_J(P) \times\cal C_J \to E_J(P)\, ,
\end{equation}
where $\cal C_J := \Sigma^\infty C_J(N)$ is a spectrum with
$\Sigma_j$-action (recall that $C_J(N)$ is the total homotopy cofiber of the $j$-cube $E_\bullet(N)$; the action of $\Sigma_j$ arises from the evident action of $\Sigma_j$ on the cube).
Note that $\Sigma_j$ acts diagonally on 
$E_J(P) \times\cal C_J$. When considered unequivariantly,
\eqref{eqn:nottrivial} is a trivial fiberwise spectrum.

Then the obstruction $e'(x)$ may be interpreted as an element
of the equivariant cohomology group
\[
H^0_{\text{cs},\Sigma_j}(E_J(P); \cal C_J)\, ,
\]
or alternatively, as an element of the function space of compactly supported $\Sigma_j$-equivariant stable maps from $E_J(P)$ to 
$\cal C_J$.

\subsubsection{The homological invariant}
By Poincar\'e duality, there is an equivalence of spectra
\[
H^0_{\text{cs},\Sigma_j}(E_J(P);\cal C_J)
\cong
H_0^{\Sigma_j}(E_J(P); {}^{-\tau}\cal C_J)
\]
where ${}^{-\tau}\cal C_J$ is the twist of $\cal C_J$ by the inverse of the tangent bundle of $E_J(P)$ (the latter is just the restriction of the product of $j$-copies of the tangent bundle of $P$). 

\begin{defn}
\[
\mu(x) \in  H_0^{\Sigma_j}(E_J(P); {}^{-\tau}\cal C_J)
\]
be the class that corresponds to $e'(x)$ via the Poincar\'e duality isomorphism.
\end{defn}

 \begin{proof}[Proof of Theorem \ref{bigthm:embedding-sequence}]
  The procedure described above defines a function
 \[
\mu \: \pi_0(\bar E_{j-1}(P,N)) \to H_0^{\Sigma_j}(E_J(P); {}^{-\tau}\cal C_J)
 \]
such that $\mu(x) = 0$ when $x$ lifts to $\pi_0(\bar E_j(P,N))$
By Lemma \ref{lem:euclid} the converse is true
provided  $r\ge  p-1-(j-1)(n-p-2)$.
 \end{proof}

\section{Spaces of link maps \label{sec:linking}}

Given manifolds
 $P_1,\cdots, P_j$ of dimension $\dim P_i = p_i$ and a connected $n$-manifold $N$ without boundary,
a {\it link map} is a continuous map 
\[
f\:P_1 \amalg \cdots \amalg P_j \to N
\]
such that $f(P_i) \cap f(P_k) = \emptyset$ for  $i \ne k$. 
We will typically assume that $P_i$ is connected and boundaryless.
Set $\bold P:= \langle P_1,\dots,P_j\rangle$ and  write
\[
\cal L(\bold P,N)
\]
for the space link maps in the compact-open topology. 

Recall that $J =\{1,2,\cdots j\}$.
For a subset $S\subset J$, set 
\[
P_S := \coprod_{i\in S} P_i\quad \text{ and } \quad  P^{(S)} := \prod_{i\in S} P_i\, .
\]
Then to each $S\subset J$, we have a space
\[
\cal L^S(\bold P,N) 
\]
whose points are the maps 
\[
f\:P_{J} \to N
\]
such that $f(P_i) \cap f(P_k) = \emptyset$ for each pair of
distinct elements $i,k \in S$.
Note that $\cal L^{J}(\bold P,N) = \cal L(\bold P,N)$ is the space
of link maps and 
if $|S| \le 1$ then $\cal L^S(\bold P,N) = F(P_{J},N)$ is the  function space of 
maps with no constraint. The assignment
\[
S\mapsto \cal L^S(\bold P,N)
\]
is contravariant and defines a $j$-cube of spaces which we denote by
\[
\cal L^\bullet(\bold P,N)\, .
\]

\begin{rem}  \label{rem:equiv-defns} There is a related $j$-cube
\[
\cal L(\bold P_\bullet,N)
\]
whose value at $S\subset J$ is the space of link maps $f\: P_S\to N$. 
Then the evident map of $j$-cubes
\[
\cal L^\bullet(\bold P,N)\to \cal L(\bold P_\bullet,N)
\]
is $\infty$-cartesian because for each $S$ we have a homotopy fiber sequence
\[
F(P_{(J \setminus S)},N) \to \cal L^S(\bold P,N)\to \cal L(\bold P_S,N)\, ,
\]
and the $j$-cube $S\mapsto F(P_{(J \setminus S)},N)$ is $\infty$-cartesian if $j > 1$.
\end{rem}

\subsection{Homotopy coherent Brunnian links}
Henceforth we fix an embedding
\[
J \subset N
\]
and identify $J$ with its image.
Let $c\: \amalg_i P_i \to N$ be the link map which sends $P_i$ to $i$. Call $c$
the {\it trivial link map}. Then $c$ equips $\cal L^\bullet(\bold P,N)$ with the structure of
a $j$-cube of based spaces. If $n \ge 2$, then
the component of the basepoint is independent
of the choice of embedding $J \subset N$.

\begin{rem} Milnor \cite{Milnor} considers the case of link maps 
$f\: \amalg_{i=1}^j P_i \to N$
in euclidean space $N= \Bbb R^3$ in which each $P_i$ is a circle $S^{1}$. Milnor defines
$f$ to be  ``trivial'' if there is an extension of $f$ to a link map
$\amalg_i D^{2} \to \Bbb R^3$. Note that $f$ is trivial in Milnor's sense if and only if
$f$ is homotopic through link maps
to the trivial link map $c$.
\end{rem}

\begin{defn} The space of {\it homotopy coherent Brunnian link maps} 
\[
\cal B(\bold P,N)
\]
is the total homotopy fiber of the $j$-cube of based spaces $\cal L^\bullet(\bold P,N)$.
\end{defn}

\begin{rems} By Remark \ref{rem:equiv-defns}, an equivalent definition up to homotopy of  
$\cal B(\bold P,N)$ is given by taking the total homotopy fiber
of the $j$-cube $\cal L(\bold P_\bullet,N)$.

A point of
 $\cal B(\bold P,N)$ is given by data consisting of
  a link map $f\: P_{J} \to N$ together with a 
homotopy coherent set of rules which to each $S\subsetneq J$ associates
a path from the associated point of $\cal L^S(\bold P,N)$ to the basepoint. 

By contrast, Milnor \cite[\S5]{Milnor}  defines a link map 
$f\: \amalg_{i=1}^{j} S^{1} \to \Bbb R^3$ to be {\it almost trivial} if every proper sublink map of $f$ is trivial.\footnote{Subsequent authors call Milnor's 
notion of almost trivial link map
a Brunnian link map. The earliest reference employing this
language seems to be \cite{Debrunner}.} If $j \ge 4$ then this  notion of Brunnian
fails to be homotopy coherent.
Thus a homotopy coherent Brunnian link map  gives an almost trivial link map but 
not conversely. 

Note that there is an evident map 
\[
\cal B(\bold P,N) \to \text{fiber}(\cal L^J(\bold P,N) @>>> \prod_{i=1}^j \cal L^{J_i} (\bold P,N))\, ,
\]
where $J_i = J \setminus \{i\}$, $\bold P = \langle S^{1},\cdots, S^1\rangle$ and $N = \Bbb R^3$. However, if $j \ge 4$, this map is not a weak equivalence. Milnor's almost trivial link maps 
are those link maps whose components are in the image of the displayed homotopy fiber.
\end{rems}

\begin{term} As we only consider homotopy coherent Brunnian link maps in this paper, we henceforth refer to $\cal B(\bold P,N)$ 
simply as the space of {\it Brunnian link maps}, 
despite the different usage of this term in the literature. 
\end{term}

\subsection{The invariants}
For each $S\subset J$, one has a map
\begin{equation} \label{eqn:invariant}
\cal L^S(\bold P,N) \to F(P^{(J)}, E_S(N))\, ,
\end{equation}
where the target is the function space of maps 
$P^{(J)}\to E_S(N)$. One defines \eqref{eqn:invariant}
by mapping a link map $f$ to the map 
\[
(x_1,\dots,x_j)\, \mapsto\,  \prod_{i\in S}f(x_i)\, .
\]  

\begin{rem} When $S= J$, the map \eqref{eqn:invariant}
is   Koschorke's {\it $\kappa$-invariant}
 $\cal L(\bold P,N) \to F(P^{(J)}, E_J(N))$.
 \end{rem}

If we let $S$ vary, \eqref{eqn:invariant} defines a map of $j$-cubes of based spaces
\begin{equation} \label{eqn:invariant2}
\cal L^\bullet(\bold P,N)\to F(P^{(J)}, E_\bullet(N))\, .
\end{equation}

\begin{rem} For $S\subset J$, let $N^J(S)$ be the 
space of $j$-tuples $x \in N^J$ such that the image of $x$
under the projection $N^J \to N^S$ lies in the subspace
$E_S(N) \subset N^S$ (here $N^S := F(S,N)$). In other words, there is a pullback diagram
\[
\xymatrix{
N^J(S) \ar[r] \ar[d] & N^J \ar[d]\\
E_S(N) \ar[r] & N^{S}\, .
}
\]
The collection $\{N^J(S)\}_{S\subset J}$ forms both a stratification of $N^J$ as well
as a $j$-cube of based spaces. 

The operation $S\mapsto F(P^{(J)}, N^J(S))$ is a $j$-cube of based spaces which  we denote by $F(P^{(J)}, N^J(\bullet))$.
Then we have a commutative diagram of $j$-cubes
\begin{equation} \label{eqn:meta-diagram}
\xymatrix{
\cal L^\bullet(\bold P,N)\ar[r] \ar[d]  & F(P^{(J)}, N^J(\bullet)) \ar[d] \\
\cal L(\bold P_\bullet,N) \ar[r] &  F(P^{(J)}, E_\bullet(N)) 
}
\end{equation}
in which the vertical maps form $\infty$-cartesian $(j+1)$-cubes
(even more is true if $N$ happens to be contractible: in this case the vertical maps
are objectwise weak equivalences of $j$-cubes).
The map  \eqref{eqn:invariant2} is just the composition of the maps in diagram \eqref{eqn:meta-diagram}.

The top horizontal map of diagram \eqref{eqn:meta-diagram} can be viewed as 
a kind of {\it coassembly map} which records the
passage from global to local linking data.
More precisely, set $\bold J := \langle 1,2,\dots j\rangle$ where
we think of $i \in J$ as a manifold of dimension zero. Then by definition
\[
N^J(S) = \cal L^S (\bold J, N)\, ,
\]
and the top horizontal map of \eqref{eqn:meta-diagram}
associates to  $f\: \amalg_i P_i \to N$ the map which sends
a $j$-tuple $(x_1,\dots,x_j) \in P^{(J)}$ to the 
composed map $\amalg_i x_i \subset \amalg_i P_i \to N$.

One has a similar description of the bottom horizontal map
by reinterpreting the configuration space $E_S(N)$ as the space of link maps
$\cal L(\bold S,N)$.
\end{rem}

\begin{defn}\label{defn:phi-e}
Let 
\[
\Phi E_\bullet(N)
\] 
be the total homotopy fiber of the $j$-cube
$E_\bullet(N)$ taken with respect to the given embedding $J\to N$. 
(Alternatively, $\Phi E_\bullet(N) $ can be defined 
as the total homotopy
fiber of the $(j-1)$-cube $N_\bullet$ appearing in the proof of 
Lemma \ref{lem:cartesian-config}.)
\end{defn}

Then the map of $j$-cubes  \eqref{eqn:invariant2}
induces a map of total homotopy fibers
\begin{equation} \label{eqn:brunnian-invariant1}
\ell\:\cal B(\bold P,N) \to F(P^{(J)},\Phi E_\bullet(N))\, ,
\end{equation}
called the {\it higher unstable linking number map}.

\begin{rem}  Let $\cal O_{\bold P}$ be the partially ordered set
 given by $\bold U = \langle U_1,\dots,U_j\rangle$ in which
 $U_i$ is an open set in $P_i$,
and $\bold U \le \bold U'$ if and only if $U_i \subset U_i'$ for all $i$. Then 
\[
\bold U \mapsto \cal B(\bold U,N)
\]
defines a contravariant functor $\cal O_{\bold P}\to \cal T_\ast$. Its
multi-linearization in the sense of Weiss' manifold calculus
coincides up to homotopy with the higher 
unstable linking number map $\ell$ (cf.\ \cite{Munson-Volic}, \cite{Munson}).
\end{rem}

\begin{conjecture} \label{conj:brunnian} The 
 map $\ell$ (cf.~\eqref{eqn:brunnian-invariant1})
is $(1+\Sigma')$-connected, where 
\[
\Sigma' = \sum_i (n-2p_i-2)\, .
\]
\end{conjecture}

\begin{rem} The $j =2$ case of Conjecture \ref{conj:brunnian} is
known in the affirmative: it is the main result of \cite{GM}.
\end{rem}

 \subsubsection{The euclidean case; stabilization} 
 Assume $N = \Bbb R^n$. Then $\Phi(E_\bullet(\Bbb R^n))$
 coincides with the   
 total homotopy fiber
of the based $(j-1)$-cube 
\[
S\mapsto \Bbb R^n \setminus S
\]
for $S\subset J_1$ (cf.~the proof of Lemma \ref{lem:cartesian-config}).
 By this identification and  Corollary \ref{cor:important} applied to 
 $Q_i := \{i\} \subset \Bbb R^n$,
 we infer there is a $(j(n-2)+1)$-connected map 
 \begin{equation}\label{eqn:identify-T}
\Phi E_\bullet(\Bbb R^n) @>>> \prod_{i=1}^{(j-2)!}   
 \Omega^\infty S^{(j-1)(n-2)+1} \, ,
\end{equation}
Applying the functor $F(P^{(J)},{-})$
to \eqref{eqn:identify-T}, one obtains a map of function spaces
\begin{equation} \label{eqn:function-spaces}
F(P^{(J)}, \Phi E_\bullet(\Bbb R^n))
@>>> \prod_{i=1}^{(j-2)!}F^{\text{st}}(P^{(J)},S^{(j-1)(n-2)+1})
\end{equation}
which is $(1 + \Sigma)$-connected, where $\Sigma = \sum_{i=1}^j (n-p_i-2)$.
The composition of \eqref{eqn:brunnian-invariant1} with  
\eqref{eqn:function-spaces} defines the {\it higher stable linking number map}
\begin{equation}\label{eqn:gln}
\lambda\: \cal B(\bold P,\Bbb R^n) @>>> \prod_{i=1}^{(j-2)!}F^{\text{st}}(P^{(J)},S^{(j-1)(n-2)+1}) \, .
\end{equation}
A version of \eqref{eqn:gln} also appears in the work of Munson \cite{Munson}. Note that \cite[cor.~1.2]{Munson} gives a connectivity estimate one less than ours
(cf.~\cite[rem.~3.6]{Munson}).

\begin{ex} Let $n=j=3$ and $P_i = S^1$ for $i = 1,2,3$. Then the higher stable linking number map $\lambda$ is of the form
\[
\cal B(\bold P,\Bbb R^3) @>>> F^{\text{st}}((S^1)^{\times 3},S^3)\, .
\]
Taking  path components gives a function $\pi_0(\cal B(S^1_\bullet,\Bbb R^3))\to\Bbb Z$. This can be described as the rule which assigns
to a three component Brunnian link in $\Bbb R^3$ 
a certain Massey product in the link complement \cite{Porter}.
\end{ex}

Since $1+ \Sigma \ge 1 + \Sigma'$,
we infer that Conjecture \ref{conj:brunnian} with $N=\Bbb R^n$ is equivalent to the following.

\begin{conjecture} \label{conj:brunnian2} The  
higher stable linking number map $\lambda$ (cf.~\eqref{eqn:gln})
is $(1+\Sigma')$-connected.
\end{conjecture}

\subsection{Evidence for Conjecture \ref{conj:brunnian2}} 
In this subsection we prove Theorem \ref{bigthm:realization} 
which we submit as evidence for Conjecture \ref{conj:brunnian2}.  

As above,  $P_1,\dots, P_j$ are closed manifolds, but now we suppose
that each $P_i$ embeds in $\Bbb R^n$. In what follows, we don't
require the $P_i$ to be pairwise disjoint and
we will not need to assume that $P_1$ is 
a submanifold of $\Bbb R^n$.

Recall the fixed embedding $J \subset \Bbb R^n$.
Choose an  $n$-balls $B(i)$ containing $i\in J \setminus 1$
and assume that the collection $\{B(i)\}$ is pairwise disjoint.
Choose an embedding $P_i \subset B(i)$
for $i\ne 1$.
 Using the inclusions $B(i) \subset \Bbb R^n$, we obtain an embedding
\[
P_2\amalg \cdots \amalg P_j \subset \Bbb R^n\, .
\]

Consider the $(j-1)$-cube of function spaces
\[
S\mapsto F(P_1,\Bbb R^n\setminus P_S)\, , \qquad S\subset J_1\, .
\]
This is a based cube, where the basepoint
of $F(P_1,\Bbb R^n\setminus P_S)$ is the constant map having value
$1\in \Bbb R^n \setminus P_S$.
Consequently, the total homotopy fiber of this cube is  given by 
\begin{equation}\label{eqn:a-cube}
F(P_1,\Phi(\Bbb R^n \setminus P_\bullet)) 	\, ,
\end{equation}
where now the  convention  is that 
$\Bbb R^n \setminus P_\bullet$ is the $(j-1)$-cube given by
$\Bbb R^n \setminus P_S$ in which $S$ ranges through subsets of $J_1$.

For $S\subset J_1$, consider the commutative diagram
\[
\xymatrix{
F(P_1,\Bbb R^n\setminus P_S) \ar[r] \ar[d]_{a_S} & \cal L^{S\amalg 1}(\bold P,\Bbb R^n) \ar[r] \ar[d]^{b_S} & F(P^{(J)},E_{S\amalg 1}(\Bbb R^n)) \ar[d]^{c_S} \\
F(P_1,\Bbb R^n) \ar[r] & \cal L^{S}(\bold P,\Bbb R^n) \ar[r] & 
F(P^{(J)},E_{S}(\Bbb R^n))\, .
}
\]
As $S$ varies, each of the vertical maps assembles to a morphism of $(j-1)$-cubes, i.e,
each gives a $j$-cube, respectively $a_\bullet, b_\bullet, c_\bullet$.
The $j$-cube $b_\bullet$ is just
$\cal L^{\bullet}(\bold P,\Bbb R^n)$.
Similarly,  $c_\bullet$
is the $j$-cube $F(P^{(J)},E_{\bullet}(\Bbb R^n))$. 
If we consider $a_\bullet$ as a map of $(j-1)$-cubes, 
then its target is the constant $(j-1)$-cube on the contractible space $F(P_1,\Bbb R^n)$,
in particular the target of $a_\bullet$ is $\infty$-cartesian.  
Hence, the total homotopy fiber $\Phi(a_\bullet)$ is identified with the total homotopy fiber of the source of $a_\bullet$, and the latter coincides with 
$F(P_1,\Phi(\Bbb R^n \setminus P_\bullet))$, i.e., the source of 
the map \eqref{eqn:a-cube}.  Consequently, taking the total homotopy fibers of 
$a_\bullet,b_\bullet$ and
$c_\bullet$ and composing with the map \eqref{eqn:function-spaces}
results in a commutative diagram
\small
\begin{equation}\label{eqn:composite}
\xymatrix{
F(P_1,\Phi(\Bbb R^n \setminus P_\bullet)) \ar[r] & \cal B(\bold P,\Bbb R^n) \ar[r]^-\ell 
\ar[dr]_-{\lambda}
& F(P^{(J)},\Phi E_\bullet(\Bbb R^n))\ar[d] \\
 && F^{\text{st}}(P^{(J)},S^{(j-1)(n-2)+1})}
\end{equation}
\normalsize  
such that the right vertical map is $(1+\Sigma)$-connected
(cf.~\eqref{eqn:function-spaces}).

\begin{rem} In the above, we've neglected to mention
 that the map of cubes
$a_\bullet \to b_\bullet$
isn't basepoint preserving. This means that the map doesn't 
define a map of total homotopy fibers in an obvious way.

However, the map is easily seen to be basepoint preserving up to a preferred path
(the path is defined by the radial deformation retraction of each ball $B(i)$ onto its center $i$).
It is this preferred path that enables us to define the map from the total homotopy fiber
of $a_\bullet$  to the total homotopy fiber of $b_\bullet$, which is the leftmost
map in \eqref{eqn:composite}.
\end{rem}

\begin{cl} \label{cl:key} The horizontal composite
\begin{equation} \label{eqn:map-of-claim}
F(P_1,\Phi(\Bbb R^n \setminus P_\bullet))\to 
F(P^{(J)},\Phi E_\bullet(\Bbb R^n))
\end{equation}
of diagram \eqref{eqn:composite}
is $(1 - \hat p +\Sigma)$-connected. 
\end{cl}

The claim, proved below, gives evidence for the validity of
Conjecture \ref{conj:brunnian2}: it implies that $\ell$ 
 is a retraction on homotopy in degrees 
$\le 1 - \hat p + \Sigma$ (the same is true for $\lambda$ since 
the vertical map of \eqref{eqn:composite} is 
$(1+\Sigma)$-connected). Furthermore,
we have $1 - \hat p + \Sigma \ge 1+\Sigma'$, so $\lambda$ will be
a retraction in degrees $\le 1+\Sigma'$.
Consequently, the proof of Theorem \ref{bigthm:realization} 
has  been reduced to
verification of the claim.

\begin{proof}[Proof of Claim \ref{cl:key}]
For $S\subset J_1$, consider the pullback diagram
\[
\xymatrix{
\cal E_S \ar[r] \ar[d] & E_{S\amalg 1}(\Bbb R^n) \ar[d]\\
P^{(J_1)} \ar[r] & E_S(\Bbb R^n)
}
\]
where the right vertical map is given by projection and
the bottom horizontal map is projection $P^{(J_1)} \to P^{(S)}$
followed by the inclusion $P^{(S)}\subset E_S(\Bbb R^n)$.
Observe that the fiber of $\cal E_S\to P^{(J_1)}$
at a point $(x_2,\dots,x_j)$ is given by $\Bbb R^n\setminus \{x_i\}_{i\in S}$.

The map  $P^{(J_1)} \to E_S(\Bbb R^n)$ factors through the contractible space $B^{(J_1)} := 
\prod_i B(i)$, so the fibration $\cal E_S \to P^{(J_1)}$ is trivializable. Let $\Gamma(\cal E_S)$
be the space of sections of $\cal E_S\to P^{(J_1)}$.
Define a map 
\[
\Bbb R^n \setminus P_S \to\Gamma(\cal E_S)
\] 
by sending a point
$z\in \Bbb R^n \setminus P_S $ to the section given  by
$(x_2,\dots,x_j) \mapsto z$. This makes sense since $z$ also lies
in $\Bbb R^n\setminus \{x_i\}_{i\in S}$.

As $S$ varies we obtain a map of $J_1$-cubes
\begin{equation} \label{eqn:better-description-claim-map}
\Bbb R^n \setminus P_\bullet \to \Gamma(\cal E_\bullet)\, ,
\end{equation}
and applying the functor $F(P_1,{-})$ to 
the induced map of total homotopy fibers of 
\eqref{eqn:better-description-claim-map} yields the map of the claim. 

Hence, it suffices to prove that
 $\eqref{eqn:better-description-claim-map}$ is 
 $(1+\mu_2+\Sigma_2)$-cartesian 
 where
 \begin{equation} \label{eqn:mu-sigma}
\mu_2 := \min_{2\le i \le j} (n-p_i-2)\, , \qquad \Sigma_2 := \sum_{i = 2}^j (n-p_i-2)\, ,
\end{equation}
since $F(P_1,{-})$ reduces
connectivity by $p_1$ and 
\[
1+\mu_2 + \Sigma_2 - p_1 = 1 -\hat p+\Sigma\, .
\]
We will explain the proof when $2\le j \le 3$.
The remaining cases are analogous to the case $j=3$ and we will
leave them for the reader to verify. 

When $j=2$, it is readily checked that the statement to be proved amounts to the assertion that the map
\[
\Bbb R^n\setminus P_2 \to F(P_2,S^{n-1})
\]
given by 
\[
z\mapsto (x\mapsto \frac{x-z}{|x-z|})
\] is 
$(1+2(n-p_2-2))$-connected. This follows  from the commutative diagram
\[
\xymatrix{
\Bbb R^n\setminus P_2\ar[r]\ar[d] & F(P_2,S^{n-1}) \ar[d]\\
\Omega^\infty \Sigma^\infty (\Bbb R^n\setminus P_2) \ar[r] &
F^{\text{st}}(P_2,S^{n-1})\, ,
}
\]
where the left vertical map is $(1+2(n-p_2-2))$-connected by 
the Freudenthal suspension theorem, the right vertical map
is $(1-p+2(n-2))$-connected, also by the Freudenthal suspension theorem, and the lower horizontal map is a homotopy equivalence
by Spanier-Whitehead duality.

When $j=3$ one proceeds as follows:  we think of square $\Bbb R^n \setminus P_S$ for $S\subset \{2,3\}$ as defining an isotopy functor 
$\phi\: \cal O_{P_2} \times \cal O_{P_3} \to T_\ast$ which assigns
to an open set $U\subset P_2$ and an open set $V\subset P_3$, the total homotopy fiber of the  square
\[
\xymatrix{
U^* \cap V^* \ar[r]\ar[d] & V^\ast\ar[d] \\
U^\ast \ar[r] & \Bbb R^n\, ,
}
\]
where  $A^\ast$ denotes the complement of $A\subset \Bbb R^n$.
Similarly one has 
an isotopy functor
 $\phi^\sharp\: \cal O_{P_2} \times \cal O_{P_3} \to T_\ast$ be
associated with the total homotopy fiber of the square $S\mapsto \Gamma(\cal E_S)$. In fact, the latter is easy to identify: it is given by 
\[
(U,V)\mapsto F(U \times V, S^{n-1}\flat S^{n-1})\, ,
\]
where $S^{n-1}\flat S^{n-1}$ is the total homotopy fiber of the 
wedge square on $S^{n-1}$. The natural map
\begin{equation} \label{eqn:coassembly}
\phi(U,V) \to \phi^\sharp(U,V)
\end{equation}
is a kind of bilinearization (or coassembly) in the sense that 
\begin{itemize}
\item its value when $U$ and $V$ are open balls is a 
homotopy equivalence;
\item $\phi^\sharp(U,V)$ is linear in each variable in the sense of isotopy calculus. 
\end{itemize}
Furthermore, \eqref{eqn:coassembly} is initial
with respect to these properties.  On the other hand,  Corollary 
\ref{cor:total-fiber}  (cf.~Corollary \ref{cor:important} and Remark \ref{rem:functoriality})
defines a natural transformation
\begin{equation} \label{eqn:natural-transform}
\phi(U,V) \to \Omega^\infty\Sigma^\infty (S^{-1}\smsh U^\ast \smsh V^\ast)
\end{equation}
whose connectivity can be described as follows:
if $U$ is a tubular neighborhood of a closed manifold of dimension $k_1$ and $V$ is a tubular neighborhood of a closed manifold of dimension $k_2$, then \eqref{eqn:natural-transform}
is $(1 +\min{(n-k_1-2,n-k_2-2)} + \sum (n-k_i-2))$-connected.
In particular, it is $(3n-5)$-connected when $U$ and $V$ are balls.

The functor $(U,V) \mapsto \Omega^\infty (S^{-1} \smsh U^\ast \smsh V^\ast)$
is also bilinear. In fact, by Spanier-Whitehead duality 
it is is naturally equivalent to the functor $\psi$ given by
\[
(U,V) \mapsto  F(U \times V, \Omega^\infty\Sigma^\infty(S^{2n-3}))\, .
\]
As $\phi\to\phi^\sharp$ is initial in the homotopy category of functors,
there is a natural transformation 
\begin{equation}\label{eqn:factorize-nt}
\phi^\sharp\to \psi
\end{equation}
that yields a factorization 
 $\phi\to \phi^\sharp\to \psi$. Clearly,
  \eqref{eqn:factorize-nt} is induced by a  map
of spaces $S^{n-1}\flat S^{n-1} \to \Omega^\infty\Sigma^\infty(S^{2n-3})$.
Furthermore, it is automatic that the map
 $\phi^\sharp(U,V)\to \psi(U,V)$
is  $(3n-5)$-connected when $U$ and $V$ are balls.

It follows that the map
$\phi^\sharp(P_2,P_3) \to \psi(P_2,P_3)$ is $(3n-5-p_2-p_3)$-connected. As
$3n-5-p_2-p_3$ is strictly larger than $1+\mu_2+\Sigma_2$, it
follows that the map $\phi(P_2,P_3) \to \phi^\sharp(P_2,P_3)$ is
$(1+\mu_2+\Sigma_2)$-connected, as was to be shown. \end{proof}

\begin{ex} Let $\bold P = \langle S^1,\dots, S^1\rangle$ be an ordered
$j$-tuple of circles and
let $n =3$. By Theorem \ref{bigthm:realization},
\[
\pi_0(\lambda)\: \pi_0(\cal B(\bold P,\Bbb R^3)) \to \prod_{i=1}^{(j-2)!} \Bbb Z
\]
is surjective.
We conjecture that  $\pi_0(\lambda)$
coincides with Milnor's $\mu$-invariants \cite[\S5]{Milnor}
on the set of (classical) Brunnian link maps.
\end{ex}  

\subsection{Postscript: the two
component case.} When $j=2$ there is some additional evidence 
for Conjecture \ref{conj:brunnian2} with the numerical improvements suggested by 
Theorem \ref{bigthm:realization}. Let $\bold P = \langle P,Q\rangle$, with   $p := \dim P$ and $q:= \dim Q$. In this situation 
$\lambda$ is the classical stable linking pairing
\begin{equation} \label{eqn:linking-number-classical}
\cal L(\langle P,Q\rangle,\Bbb R^n) \to F^{\text{st}}(P\times Q,S^{n-1})\, ,
\end{equation}
which associates to a link map $f\amalg g\: P\amalg Q \to \Bbb R^n$ the map
\[
(x,y) \mapsto \frac{f(x) - g(y)}{|f(x)-g(y)|}\, .
\]
On path components the above gives a function of pointed sets
\begin{equation}
\alpha\: \pi_0(\cal L(\langle P,Q\rangle,\Bbb R^n)) \to \{P_+\smsh Q_+,S^{n-1}\}\, ,
\end{equation}
where we have identified the set of path components of
$F^{\text{st}}(P\times Q,S^{n-1})$ with the abelian group of stable homotopy classes of based maps
$P_+\smsh Q_+\to S^{n-1}$.

Suppose $A$ and $B$ are pointed sets. We denote the basepoint in each case by $\ast$.
A basepoint preserving map $h\: A\to B$ is said to be {\it weakly injective} if there 
are no nontrivial solutions to the equation 
$h(x) = \ast$. If $h$ is a homomorphism of groups, then weak injectivity implies
injectivity (compare \cite[lem.~1.1]{Habegger-Kaiser}).

\begin{prop} \label{prop:Good-Munson-improved} Assume that $Q\subset \Bbb R^n$ is a submanifold of codimension $\ge 3$.
Then the function $\alpha$ is a surjection on path components
if $2n-2q-p-3 \ge 0$. Furthermore, if $2n-2q-p-3 > 0$ then
$\alpha$ is weakly injective.
\end{prop}

\begin{rems} (1). Proposition \ref{prop:Good-Munson-improved} gives a better estimate than
\cite{GM}, but at the expense of an additional hypothesis on $Q$.
\smallskip

\noindent (2). The number $2n-2q-p-3$ may be rewritten in the form $1-q + \Sigma$, where $\Sigma = (n-p-2) + (n-q-2)$. This is the  number of  Theorem \ref{bigthm:realization} when $j=2$.
Hence,  only  weak injectivity needs to be verified.
\smallskip

\noindent (3). Proposition 
\ref{prop:Good-Munson-improved} 
suggests that the connectivity estimate of Conjecture \ref{conj:brunnian2} might be improved to 
$1-\hat p + \Sigma$ under the additional assumption that
$P_2,\dots,P_j \subset \Bbb R^n$ are submanifolds of codimension $\ge 3$.
\smallskip

\noindent (4). Proposition \ref{prop:Good-Munson-improved} delivers more information
in the spherical case  $P = S^p$ and $Q= S^q$ with $q\le n-3$. Then
 $\pi_0(\cal L(\langle S^p,S^q\rangle,\Bbb R^n))$ possesses a group structure 
 (see \cite{Scott}, \cite[p.~765]{Koschorke3})  and the function $\alpha$ becomes a homomorphism. Consequently,
 weak injectivity implies injectivity and we recover \cite[p.~190]{Scott}.
 We infer that Proposition \ref{prop:Good-Munson-improved} implies that
  $\alpha$ is an isomorphism
 when $2n-2q-p-3 > 0$. According to \cite[thm.~1.1]{Habegger-Kaiser}, in the spherical case
 $\alpha$ is actually 
 an isomorphism if $3n - 2q - 2p - 4> 0$ and $p, q \ge 1$.
\end{rems}

\begin{proof}[Proof of 
 Proposition \ref{prop:Good-Munson-improved}]
 As pointed out above, we only need to verify the last part of the statement.
 Let 
\[
x:= f\amalg g \in \cal L(\langle P,Q\rangle),\Bbb R^n)
\]
be any point. We can assume without loss in generality that  $f\: P \to \Bbb R^n$ is a smooth map.
We first show how to find a path in $\cal L(\langle P,Q\rangle,\Bbb R^n)$ from $x$ to $x' = (f,h)$
in which $h$ is a smooth embedding. 
It then suffices to prove that if the stable linking number of $x'$ is trivial then
the map $f\: P \to \Bbb R^n\setminus h(Q)$ is null-homotopic.

Consider the commutative square
\[
\xymatrix{
E(Q,\Bbb R^n \setminus f(P)) \ar[r] \ar[d] & F(Q,\Bbb R^n \setminus f(P))\ar[d]\\
E(Q,\Bbb R^n) \ar[r] & F(Q,\Bbb R^n)\, ,
}
\]
in which $E({-},{-})$ denotes the space of embeddings. By Lemma \ref{lem:GK-improved} below, 
the  square is $(2n-2q-p-3)$-cartesian.
If particular, if we use the preferred basepoint of $E(Q,\Bbb R^n)$, it follows that
when $2n-2q-p-3\ge 0$, we
can  find an isotopy of the submanifold $Q\subset \Bbb R^n$  to an embedding 
$h\: Q\to \Bbb R^n \setminus f(P)$ such that the underlying map of this embedding is homotopic
to the map $g\: Q\to \Bbb R^n \setminus f(P)$. Then  $x' = (f,h)\in \cal L(\langle P,Q\rangle,\Bbb R^n)$ 
is in the same path component as $x$. 

But as we've seen above, the composition
\[
F(P,\Bbb R^n \setminus h(Q)) \to \cal L(\langle P,Q\rangle,\Bbb R^n) \to 
F^{\text{st}}(P\times Q,S^{n-1})
\]
is $(2n-2q-p-3)$-connected. In particular, if $2n-2q-p-3 > 0$ then
the triviality of the stable linking number of
$x'$ implies that the map $P \to 
\Bbb R^n \setminus h(Q)$ is null-homotopic.
\end{proof}

The following result was used in the proof of Proposition \ref{prop:Good-Munson-improved}:

\begin{lem} \label{lem:GK-improved} 
Assume $N$ is a connected smooth $n$-manifold, and let $P$ and $Q$ be closed smooth manifolds
of dimensions $p$ and $q$. Assume $q \le n-3$.
Let $f\: P \to N$ be a smooth map. 
Then the square
\[
\xymatrix{
E(Q,N \setminus f(P)) \ar[r] \ar[d] & F(Q,N \setminus f(P))\ar[d]\\
E(Q,N) \ar[r] & F(Q,N)
}
\]
is $(2n-2q-p-3)$-cartesian.
\end{lem} 

\begin{rem} When $f$ is an embedding, this amounts to the $j=2$ case of 
\cite[th.~E]{GK2}.
\end{rem}

\begin{proof} (Sketch). The argument was communicated to us by Tom Goodwillie.
If we replace embeddings with immersions, then the analogous diagram
is $\infty$-cartesian by Smale-Hirsch theory (in this instance we only need to assume $q\le n-1$).  Hence it suffices to show that the square
\[
\xymatrix{
E(Q,N \setminus f(P)) \ar[r] \ar[d] & I(Q,N \setminus f(P))\ar[d]\\
E(Q,N) \ar[r] & I(Q,N)\, ,
}
\]
is $(2n-2q-p-3)$-cartesian, where $I({-},{-})$ denotes the space of immersions.

The proof then proceeds by comparing the homotopy fibers of the horizontal maps of the square. The map  $N\setminus f(P) \to N$ is $(n-p-1)$-connected by transversality.
If $q\le n-3$, then the Goodwillie-Weiss embedding calculus
applied to the embedding spaces $E(Q,N \setminus f(P))$ and $E(Q,N)$ give towers for these homotopy fibers, where the first non-trivial layer is in degree $j \ge 2$. 
The homotopy theoretic model for these layers provided by 
 \cite{Weiss} implies that the map of the $j$-th layers is  
$(2n-2q-p-3)$-connected for all $j$. The conclusion then follows from the five lemma.
\end{proof}

\end{document}